\documentclass[11pt]{amsart}
\usepackage{amssymb}
\usepackage{amsmath}
\usepackage{amsfonts,latexsym,amstext}
\usepackage[foot]{amsaddr}

\usepackage{color}
\usepackage{bbm}

\newtheorem{theorem}{Theorem}[section]
\newtheorem{definition}[theorem]{Definition}

\newtheorem{lemma}[theorem]{Lemma}
\newtheorem{proposition}[theorem]{Proposition}
\newtheorem{corollary}[theorem]{Corollary}
\newtheorem{remark}[theorem]{Remark}
\newtheorem{assump}[theorem]{Assumptions}

\newcommand{\R}{{\mathbb R}}
\newcommand{\C}{{\mathcal C}}

\newcommand{\Z}{{\mathbb{Z}}}

\newcommand{\N}{{\mathbb N}}

\newcommand{\F}{{\mathcal F}}

\newcommand{\E}{{\mathbb E}}

\newcommand{\B}{{\mathcal B}}

\def\R{\mathbb{R}}
\def\N{\mathbb{N}}
\def\E{\mathbb{E}}

\def\F{\mathcal F}

\title[LLN in $L^2$ for supercritical branching Markov processes]{On laws of large numbers in $L^2$ for supercritical branching Markov processes beyond $\lambda$-positivity}

\author{Matthieu Jonckheere$^*$ and Santiago Saglietti$^\dagger$}

\address[$*$]{Departamento de Matem\'atica, Universidad de Buenos Aires, 1428 Buenos Aires, Argentina.\newline email:\,{\tt mjonckhe@dm.uba.ar.}}

\address[$\dagger$]{Faculty of Industrial Engineering and Management, Technion, 3200003 Haifa, Israel.\newline email:\,{\tt saglietti.s@technion.ac.il.}} 

\begin{document}
\maketitle

\begin{abstract}
We give necessary and sufficient conditions for laws of large numbers to hold in $L^2$ for the empirical measure of a large class of branching Markov processes, including $\lambda$-positive systems but also some $\lambda$-transient ones, such as the branching Brownian motion with drift and absorption at $0$. This is a significant improvement over previous results on this matter, which had only dealt so far with $\lambda$-positive systems. Our approach is purely probabilistic and is based on spinal decompositions and many-to-few lemmas. In addition, we characterize when the limit in question is always strictly positive on the event of survival, and use this characterization to derive a simple method for simulating (quasi-)stationary distributions.
\end{abstract}

\section{Introduction}

Since the late seventies there has been an intensive effort of research dedicated to proving laws of large numbers for branching processes. Given a (continuous time) branching process $\xi=(\xi_t)_{t \geq 0}$ whose particles (or individuals) live on some measurable space $(J,\B_J)$, by a law of large numbers we shall understand the following: there exists a nonempty class $\mathcal{C} \subseteq \B_J$, a  measure $\nu$ on $(J,\B_J)$ and a random variable $D_\infty$ such that for all pairs $B,B' \in \mathcal{C}$ with $\nu(B') \neq 0$
\begin{equation} \label{eq:LLN}
\frac{\xi_t(B)}{\E(\xi_t(B'))} \longrightarrow \frac{\nu(B)}{\nu(B')} \cdot D_\infty
\end{equation} where, for $A \in \B_J$, we denote by $\xi_t(A)$ the number of particles of $\xi$ inside the set $A$ at time $t$. 
The earliest results in this regard can be found in \cite{moy1967,watanabe1967,asmussen1976,athreya2}. Later, the renovated approach in \cite{kurtz1997} introducing the notion of spine for the \mbox{branching process} sprouted a multitude of new results, see for instance \cite{englander2010,harris2014,harris2015}. In the recent works \cite{ChenShio2007,chen2016}, functional analytic methods were used to obtain results in the setting of branching symmetric Hunt processes. 
Laws of large numbers were also investigated in the related context of superprocesses, see  ~\cite{EngTur2002,englander2006law,Chen2015,ChenRenWang2008,LiuRenSong2013}. For a more detailed overview of past results and recent developments on these matters, we refer the interested reader to \cite{englander2014spatial,englandersurvey}. See also \cite[Section 2.5]{JonckSag}.

Whenever the convergence in \eqref{eq:LLN} is understood in the $L^2$-sense, all previous results so far require the branching process $\xi$ to be $\lambda$-\textit{positive}. Essentially, $\lambda$-positivity means that the motion of a certain spine describing the genealogy of the branching process (which is sometimes referred to in the literature as the \textit{immortal particle}) is positive recurrent, a property which proves crucial in all of the approaches developed until now. 
However, recent questions steaming from particle systems demand for a better understanding of branching processes which fail to be $\lambda$-positive. Indeed, there exists a large body of literature studying empirical measures of population models with mutations and selection  \cite{BD,BD2,BHM,BG,FM,AFGJ,M,DR,GJ}, for which hydrodynamic limits were obtained on finite time windows \cite{V,AFG,BG,DR}. However, it is still an open problem in many of these systems to obtain scaling limits of the empirical measure in its stationary regime, see \cite{AFGJ,GJ,M}. Different couplings with branching Markov processes have been proposed to study this problem, but typically the resulting branching process is not $\lambda$-positive, see \cite{AFGJ,M}. Thus, this highlights the need of a convergence theory for empirical measures of branching Markov processes going past this assumption of $\lambda$-positivity. 
A canonical example appearing in this context is the Branching Brownian Motion with drift and absorption at $0$. Kesten introduced this model in his paper \cite{kesten1978} from 1978 and stated there that a strong law of large numbers holds for this process whenever it is supercritical, although he did not provide a proof of this fact nor did he made any similar assertions regarding $L^2$-convergence. The validity of Kesten's claim was only very recently proved in \cite{louidor2017strong}, almost forty years later, building upon the results from the present article. As for the case with no drift, which is also not $\lambda$-positive, a strong law was first established in \cite{watanabe1967} and its weak analogue for super-Brownian Motion was then obtained in \cite{englander2009law}, though neither of these works address the question of convergence in $L^2$.

Once a law of large numbers as in \eqref{eq:LLN} is established, it is then of particular interest to determine whether $D_\infty > 0$ holds almost surely in the event the branching process survives forever (assuming it is supercritical, so that there is a positive probability that this occurs). Indeed, assuming that $J \in \mathcal{C}$ and $\nu(J)=1$, it is usually simple to check that the former statement implies the following convergence for the empirical measure associated with $\xi$: conditionally on the event of survival, for all $B \in \mathcal{C}$ one has that
\begin{equation} \label{eq:LLN2}
\frac{\xi_t(B)}{\xi_t(J)} \longrightarrow \nu(B).
\end{equation} In turn, we see that whenever $D_\infty > 0$ holds almost surely in the event of survival, we obtain a more complete description of the asymptotic behavior of the branching process as $t \rightarrow +\infty$:
\begin{itemize}
	\item [i.] For any $B \in \mathcal{C}$ with $\nu(B) \neq 0$, the number of particles $\xi_t(B)$ in $B$ grows like its expectation (times $D_\infty$, which acts as a random scale constant).
	\item [ii.] If $J \in \mathcal{C}$ and $\nu(J)=1$, the proportion of particles $\frac{\xi_t(B)}{\xi_t(J)}$ inside any given set $B \in \mathcal{C}$ behaves asymptotically as $\nu$.
\end{itemize}
Furthermore, we also see that \eqref{eq:LLN2} yields a simple and direct method to simulate the distribution $\nu$ which, in many cases, might not be known explicitly. However, showing that $D_\infty > 0$ in the event of survival is a subtle question (and, in fact, it is not always true, see Section \ref{sec:tou} for example) and the literature addressing this matter is very limited when dealing with more involved situations with absorption and/or infinite state spaces, see \cite{harris2009,harris2006v1} for some specific examples.

Our contribution in this article is two-fold. First, we derive a necessary and sufficient condition for laws of large numbers to hold in $L^2$ for a wide class of supercritical branching Markov processes having a constant branching rate and an offspring distribution with a finite second moment, which includes many $\lambda$-positive systems but also other highly relevant examples without this property.
More precisely, we show that, whenever the immortal particle process mentioned above (which is well-defined even for non-$\lambda$-positive systems) has a distribution which is of regular variation as $t$ tends to infinity (see Section \ref{sec:main1} for details), then \eqref{eq:LLN} holds in $L^2$ if and only if a specific additive martingale associated with $\xi$ is bounded in $L^2$. Furthermore, we show that, in the latter case, $D_\infty$ is precisely the $L^2$-limit of this martingale. Finally, we also obtain an explicit formula for the asymptotic variance of this martingale, so that one can determine whether it is indeed bounded in $L^2$ by performing a direct computation.
Our approach is purely probabilistic and based solely on simple spinal decomposition techniques, namely the ``many-to-one'' and ``many-to-two'' lemmas, which allow us to effectively control the particle correlations as time tends to infinity. 
 
We then focus on studying conditions which guarantee that $D_\infty > 0$ almost surely upon survival, i.e. $P(D_\infty>0|\text{survival})=1$ (notice that this is in fact stronger than just non-degeneracy of $D_\infty$, which only amounts to having $P(D_\infty>0|\text{survival})>0$). We show that, whenever \eqref{eq:LLN} holds in $L^2$, the equality $P(D_\infty>0|\text{survival})=1$ is equivalent to the process $\xi$ being \textit{strongly supercritical}: this means that, in the event of survival, particles of the process can never accumulate all together on the boundary of the state space. 
 This notion of strong supercriticality is related to the concept of \textit{strong local survival} studied in \cite{bertacchi2014} and references mentioned therein, although we are not aware of any previous connections made between this and the strict positivity of $D_\infty$ on survival. 

Finally, we illustrate our results through a series of examples. First, we show that any \mbox{$\lambda$-positive} system whose associated immortal particle relaxes sufficiently fast to equilibrium (i.e. it admits a geometric Lyapunov functional growing sufficiently fast at infinity) verifies our hypotheses and thus satisfies a law of large numbers in $L^2$.
As a matter of fact, we also show that, if in addition this Lyapunov functional does not grow too fast, then the almost sure convergence holds as well. 
Afterwards, we use this to obtain laws of large numbers for several classic $\lambda$-positive systems: branching ergodic motions, branching Galton-Watson processes, branching contact processes and branching inward/outward Ornstein-Uhlenbeck processes. Finally, we study the emblematic case of the Branching Brownian Motion with drift and absorption at $0$ presented in \cite{kesten1978}, which is not $\lambda$-positive. For this system we completely characterize the region of parameters for which a law of large numbers holds in $L^2$ and, in particular, show that it is strictly smaller than the region of parameters for which the process is supercritical.    

The rest of the article is structured as follows. In Section 2 we describe the basic setup and notation, and then state our main results. 
Section 3 discusses several examples and applications. In Section 4 we recall the many-to-few lemma, which constitutes one of the main tools of our analysis, while Sections 5, 6, 7 and 8 contain the proofs of our main results. In order to shorten the exposition, several details of the proofs and examples as well as a detailed review of previous results are available in an extended previous preprint version, \cite{JonckSag}.

\section{Preliminaries and main results}\label{sec:results}

\subsection{Preliminaries}

Let $X=(X_t)_{t \geq 0}$ be a homogeneous Markov process with cadlag trajectories on some metric space $\overline{J}$. We will assume throughout that:
\begin{enumerate}
	\item [$\bullet$] $X$ is allowed to have \textit{absorbing states}, i.e. $x \in \overline{J}$ such that, whenever $X_t = x$ for some $t$, one has $X_{s}=x$ for all times $s>t$.
	\item [$\bullet$] The set $\partial_* \overline{J}$ of all absorbing states of $X$ belongs to $\B$, the Borel $\sigma$-algebra of $\overline{J}$.
	\item [$\bullet$] $J:=\overline{J}-\partial_* \overline{J}$ is locally compact and separable. 
\end{enumerate}  
Now, consider the following branching dynamics:
\begin{enumerate}
	\item [i.] The dynamics starts with a single particle, located initially at some $x \in \overline{J}$, whose position evolves randomly according to $\mathcal{L}$, the infinitesimal generator of $X$. 
	\item [ii.] This initial particle branches at rate $r > 0$, dying and being replaced at its current position by an independent random number of particles $m$, taking values in $\N_0$.
	\item [iii.] Starting from their birth position, now each of these new particles independently mimics the same stochastic behavior of its parent.
	\item [iv.] If a particle has $0$ children, then it dies and moves to a graveyard state $\Delta$ forever.
\end{enumerate}
Given any time $t \geq 0$, for each particle $u$ present in the dynamics at time $t$ we write $u_t$ to indicate its position at time $t$. Also, we let $\overline{\chi}_t$ denote the collection of particles in the branching dynamics which are alive at time $t$, i.e. $u_t \notin \Delta$. We identify $\overline{\chi}_t$ with a finite measure $\chi_t$ on $(\overline{J},\B)$ by setting 
$$
\chi_t := \sum_{u \in \overline{\chi}_t} \delta_{u_t}.
$$ Furthermore, let $\overline{\xi}_t$ denote the collection of particles $u$ in $\overline{\chi}_t$ which have not been absorbed yet, i.e. such that $u_t \in J$, and define its induced measure $\xi_t$ on $(J,\B_J)$ as
$$
\xi_t = \sum_{u \in \overline{\xi}_t} \delta_{u_t}.
$$ Finally, we write $|\xi_t|:=\xi_t(J)$ for the total mass of $\xi_t$, i.e. the number of living particles at time $t$ which have not been absorbed yet, and define the \textit{empirical measure} $\nu_t$ as
$$
\nu_t := \frac{1}{|\xi_t|} \cdot \xi_t 
$$ with the convention that $\infty \cdot 0 = 0$, used whenever $|\xi_t|= 0$.

Throughout the rest of the article we will use the subscript $x$ in the notation, e.g. in $P_x$ or $\E_x$, to indicate that the process involved in the corresponding probability or \mbox{expectation starts at $x$.} Similarly, the superscript $x$, e.g. in $\xi_t^{(x)}$ or $X_t^{(x)}$, indicates the corresponding process starts at $x$. 

\subsection{A necessary and sufficient condition for laws of large numbers in $L^2$}\label{sec:main1}

We now begin to present and discuss our main results. 
Before we can do so, however, we must introduce some assumptions on our branching dynamics. Our initial assumptions on the underlying motion $X$ are the following.

\begin{assump} \label{assumpG} $\,$
	\begin{enumerate}
		\item [A1.] There exists $\lambda \geq 0$ and a nonnegative $\B$-measurable function $h:\overline{J} \rightarrow \R_{\geq 0}$ such that:
		\begin{itemize}
			\item [i.] $h(x)=0$ if and only if $x \in \partial_* \overline{J}$.
			\item [ii.] For every $x \in J$ the process $M^{(x)}=(M^{(x)}_t)_{t \geq 0}$ given by the formula
			$$
			M_t^{(x)} := \frac{h(X_t^{(x)})}{h(x)}e^{\lambda t},
			$$ is a (mean-one) square-integrable martingale, i.e. $-\lambda$ is a (right) eigenvalue of $\mathcal{L}$ with associated eigenfunction $h$ satisfying $\E_x(h^2(X_t)) < +\infty$ for all $t \geq 0$.
		\end{itemize}
		
		\item [A2.] There exists a nonempty class of subsets $\mathcal{C}_X \subseteq \B_J$ such that for each $x \in \overline{J}$ and $B \in \mathcal{C}_X$ one has the asymptotic formula
		\begin{equation}\label{A2}
		P_x( X_t \in B ) = h(x)p(t)e^{-\lambda t}(\nu(B) + s_B(x,t)),
		\end{equation} for all $t > 0$, where $\lambda$ and $h$ are those from (A1) and:
		\begin{itemize}
			\item [i.] $\nu$ is a (non-necessarily finite) measure on $(J,\B_J)$ satisfying:
			\begin{enumerate}
				\item [$\bullet$] $\nu(B) \in [0,+\infty)$ for all $B \in \mathcal{C}_X$,
				\item [$\bullet$] there exists at least one $B' \in \mathcal{C}_X$ such that $\nu(B')>0$.
			\end{enumerate}
			\item [ii.] $p(t)$ is a regularly varying function at infinity, i.e. a function $p: (0,+\infty) \rightarrow (0,+\infty)$ such that the limit
			$$
			\ell(a):=\lim_{t \rightarrow +\infty} \frac{p(at)}{p(t)} 
			$$ exists and is finite for all $a > 0$. 
			\item [iii.] $s_B(\cdot,t)$ converges to zero as $t \rightarrow +\infty$ uniformly over 
			$J_n:=\{x \in J : \frac{1}{n}\leq h(x) \leq n\}$ for each $n \in \N$.
			
			\item [iv.] There exist $t_0,\overline{s}_B > 0$ such that $\sup_{x \in \overline{J}}s_B(x,t) \leq \overline{s}_B$ for all $t > t_0$.
			\end{itemize} 
	\end{enumerate}
\end{assump}

\begin{remark}\label{rem:ip} Let us observe that if for $x \in J$ we consider the martingale change of measure $\tilde{P}_x$ (known in the literature as $h$-transform) given by 
	\begin{equation} \label{eq:mcof}
	\frac{d\tilde{P}_x}{dP_x}\bigg|_{\F_t^{(x)}} = M_t^{(x)},
	\end{equation} where $(\F_t^{(x)})_{t \geq 0}$ denotes the filtration generated by $X^{(x)}$, and also define the measure $\mu$ on $(J,\B_J)$ via the formula
	\begin{equation} \label{eq:defmu}
	\frac{d\mu}{d\nu}=h,
	\end{equation}then 
	\begin{equation}
	\label{eq:fasymp}
	P_x( X_t \in B) = \E_x (\mathbbm{1}_{B}(X_t)) = h(x)e^{-\lambda t} \tilde{\E}_x\left( \frac{\mathbbm{1}_B}{h}(X_t)\right) = h(x)p(t)e^{-\lambda t}\left( \nu(B) + s_B(x,t)\right)
	\end{equation} where $\tilde{\E}_x$ denotes expectation with respect to the measure $\tilde{P}_x$ and $s_B$ is given by 
\begin{equation}
\label{eq:defsb}
s_B(x,t):= \frac{1}{p(t)}\tilde{\E}_x\left( \frac{\mathbbm{1}_B}{h}(X_t)\right) - \nu(B) = \frac{1}{p(t)}\tilde{\E}_x\left( \frac{\mathbbm{1}_B}{h}(X_t)\right) - \mu\left(\frac{\mathbbm{1}_B}{h}\right).
\end{equation} Thus, (A2) can be reformulated as the assumption that there exists a regularly varying function $p$ at infinity such that the limit
\begin{equation}
\label{eq:plimit}
\lim_{t \rightarrow +\infty} \frac{1}{p(t)}\tilde{\E}_x\left( \frac{\mathbbm{1}_B}{h}(X_t)\right) 
\end{equation} exists for all $B \in \C_X$, is given by the non-trivial measure $\mu$ and, moreover, \mbox{is uniform over each $J_n$.} In particular, it follows that, for any $B \in \mathcal{C}_X$ such that $\nu(B) > 0$, the function 
$$
\ell^{(x)}_B(t):=\tilde{\E}_x\left( \frac{\mathbbm{1}_B}{h}(X_t)\right)
$$ is of regular variation at infinity. In conclusion, we may regard (A2) as requiring that, under $\tilde{P}_x$, the distribution of $X^{(x)}$ be (in some sense) of regular variation at infinity uniformly over each $J_n$. 
\end{remark}
	
Assumptions \ref{assumpG} are satisfied in many different situations. Indeed, as we will see in Section \ref{sec:rpos1}, the underlying motion of any $\lambda$-positive branching process (see Section \ref{sec:rpos1} for a precise definition) is a natural example of system verifying these conditions, and this already covers a wide range of possibilities: not only are ergodic motions in this category, but also certain transient systems as well as many examples of almost-surely absorbed motions having $\nu$ as their Yaglom limit. Moreover, as mentioned in the Introduction, Assumptions \ref{assumpG} are also satisfied by processes which are not $\lambda$-positive, see  \cite{watanabe1967,kesten1978} and Section \ref{sec:bbm} below for Brownian motion with nonpositive drift and absorption at $0$ and also \cite{KP1} for other examples of L\'evy processes satisfying these assumptions. 

On the other hand, we point out that the class $\mathcal{C}_X$ was introduced in Assumptions \ref{assumpG} because, in general, one cannot expect the asymptotics in \eqref{A2} to be valid for \textit{every} $B \in \B_J$, see \mbox{Section \ref{sec:tou}}. Nevertheless, even if this is not the case one can still produce convergence results which hold for all $B$ in this smaller class $\mathcal{C}_X$. Finally, we observe that in all the examples of Section \ref{sec:examples} the measure $\nu$ actually corresponds to the left eigenmeasure of the generator $\mathcal{L}$ associated to $-\lambda$. However, this fact will not be used throughout our analysis. 

Now, since we are interested in understanding the evolution in $L^2$ of the branching dynamics $\xi$ in the supercritical case in which $|\xi_t|$ remains positive for all times $t > 0$ with positive probability, we must also make the following assumptions on $m$ and $r$.

\begin{assump}\label{assumpG0} We shall assume throughout that the pair $(m,r)$ satisfies:	\begin{itemize}
		\item [I1.] $m_2:= \E(m^2) < +\infty$ and $m_1:= \E(m) > 1$.
		\item [I2.] $r(m_1-1) > \lambda$, where $\lambda$ is the parameter from Assumptions \ref{assumpG}.
	\end{itemize}
\end{assump} It follows from Assumptions \ref{assumpG0} and Lemmas \ref{lema:mt1}-\ref{lema:mt2} below that all second moments $\E_x(|\xi_t|^2)$ are well-defined for any $x \in J$ and $t \geq 0$, and also $\E_x(|\xi_t|) \rightarrow +\infty$ as $t \rightarrow +\infty$ for all $x$ as long as there exists at least one $B \in \mathcal{C}_X$ with $\nu(B)> 0$, which is precisely the case that interests us. 

Now, our first result is concerned with the $L^2$-convergence of the so-called \textit{Malthusian martingale} associated to our branching dynamics, which we define below.

\begin{definition} For any $x \in J$ we define the \textit{Malthusian martingale} $D^{(x)}=(D_t^{(x)})_{t \geq 0}$ as
	$$
	D_t^{(x)} := \frac{1}{h(x)}\sum_{u \in \overline{\xi}^{(x)}_t} h(u_t) e^{-(r(m_1-1)-\lambda)t}.
	$$
\end{definition} 
It follows from the many-to-one lemma in Section \ref{sec:mtf} and (A1) that $D^{(x)}$ is indeed a martingale. Furthermore, (A1) implies in fact that $D^{(x)}$ is square-integrable. Being also nonnegative, we know that there exists an almost sure limit $D^{(x)}_\infty$. Our first result is then the following.

\begin{theorem}\label{theo:main2} For every $x \in J$ we have that  
	$$
	\lim_{t \rightarrow +\infty} \E_x(D_t^2) = (m_2-m_1) \int_0^\infty \E_x(M_s^2) re^{-r(m_1-1)s}ds=:\Phi_x,
	$$ so that $D^{(x)}$ converges in $L^2$ to $D^{(x)}_\infty$ if and only if $\Phi_x < +\infty$. In this case, we have that
	$$
	\E_x(D_\infty)=1 \hspace{2cm}\text{ and }\hspace{2cm}\E_x(D_\infty^2)=\Phi_x.
	$$	
\end{theorem} 

We should point out that the $\Phi_x < +\infty$ condition is not trivial under Assumptions \ref{assumpG}, so that $L^2$-convergence may not always hold, see e.g. Section \ref{sec:bbm}. However, the $L^2$-convergence of $D^{(x)}$ is crucial as it dictates the validity of a law of large numbers in $L^2$ for $\xi$, as our next result shows.

\begin{theorem}\label{theo:main}
	If for $x \in J$ and $B,B' \in \mathcal{C}_X$ with $\nu(B') > 0$ we define  $W^{(x)}(B,B')=(W^{(x)}_t(B,B'))_{t \geq 0}$ by the formula
	$$
W^{(x)}_t(B,B') := \frac{\xi^{(x)}_t(B)}{\E_x(\xi_t(B'))}
$$ (which is well-defined for $t$ large enough since $\liminf_{t \rightarrow +\infty} \E_x(\xi_t(B')) > 0$ by Lemma \ref{lema:mt1} and \eqref{A2}), then the following holds:
\begin{enumerate}
	\item [i.] The sequence $W^{(x)}(B,B')$ satisfies 
	$$
	\lim_{t \rightarrow +\infty} \E_x(W^2_t(B,B'))= \left[\frac{\nu(B)}{\nu(B')}\right]^2 \Phi_x.
	$$ In particular, it is bounded in $L^2$ if and only if $\Phi_x < +\infty$. 
	\item [ii.] If $W^{(x)}(B,B')$ is bounded in $L^2$ then we have that as $t \rightarrow +\infty$
	\begin{equation} \label{eq:conv1}
	W^{(x)}_t(B,B') \overset{L^2}{\longrightarrow} \frac{\nu(B)}{\nu(B')} \cdot D_\infty^{(x)}.
    \end{equation}
    In particular, conditionally on the event $\{D^{(x)}_\infty > 0\}$, we have that as $t \rightarrow +\infty$ 
    \begin{equation} \label{eq:ks2}
    \nu^{(x)}_t(B,B'):=\frac{\xi^{(x)}_t(B)}{\xi^{(x)}_t(B')} \overset{P}{\longrightarrow} \frac{\nu(B)}{\nu(B')}.
    \end{equation}
 \end{enumerate} 
       
\end{theorem} 

We note that, due to the presence of absorption, $W^{(x)}(B,B')$ will not be a martingale in general, so that the existence of an limit in $L^2$ whenever it is bounded is by no means a trivial statement. 
Still, the main idea behind Theorem \ref{theo:main} is that, whenever $W^{(x)}(B,B')$ is $L^2$-bounded, it behaves asymptotically as the martingale $\frac{\nu(B)}{\nu(B')}\cdot D^{(x)}$ and thus it must also converge.

\subsection{Strict positivity of $D^{(x)}_\infty$ on the event of non-extinction}\label{sec:strict}

Observe that for every $x \in J$ the event $\Lambda^{(x)}:=\{D^{(x)}_\infty > 0\}$ is contained in the event of non-extinction \mbox{$\Theta^{(x)}:=\{ |\xi_t| > 0 \text{ for all }t\}$.} Ideally, we would like both events to be almost surely equal, i.e. 
\begin{equation}\label{eq:eq}
P_x( \Lambda | \Theta )=1,
\end{equation} since, in that case, Theorem \ref{theo:main} would tell us that, almost surely on the event of non-extinction, for any $B \in \mathcal{C}_X$ with $\nu(B) > 0$ the number of particles $\xi^{(x)}_t(B)$ grows like its expectation which, using the results from Section \ref{sec:mtf}, can be explicitly computed and, furthermore, that these particles distribute themselves according to $\nu$. The following result intends to give conditions under which the equality in \eqref{eq:eq} is guaranteed. It is based on the study of the moment generating operator associated to the branching dynamics, which we define now.

\begin{definition} \label{def:defG}
	Let $\mathfrak{B}$ denote the class of measurable functions $g: (J,\B_J) \to ([0,1],\B_{[0,1]})$, where $\B_{[0,1]}$ denotes the Borel $\sigma$-algebra on $[0,1]$. We define the \textit{moment generating operator} $G: \mathfrak{B} \to \mathfrak{B}$ by the formula
	$$
	G(g)(x):=\E_x \left( \prod_{u \in \overline{\xi}_1} g(u_1)\right)
	$$ with the convention that $\prod_{u \in \emptyset} = 1$, used whenever $|\xi_1| = 0$. 
\end{definition}

It is immediate to see that $\mathbf{1}$, the function constantly equal to one on $J$, is a fixed \mbox{point of $G$,} i.e. $G(\mathbf{1})=\mathbf{1}$. Furthermore, by the branching property of the dynamics one has that the functions 
\begin{equation}\label{eq:defeta}
\eta(x):= P_x( \Theta^c) \hspace{1cm}\text{ and }\hspace{1cm}\sigma(x):=P_x(D_\infty=0)
\end{equation} are also fixed points of $G$, see Proposition \ref{prop:fp} below.\footnote{The fact that both $\eta$ and $\sigma$ are indeed measurable holds if one assumes that the process $X^{(x)}$ is also measurable as a function of $x$ in a suitable manner.} Since clearly $\eta \leq \sigma$ and we also have $\sigma \neq \mathbf{1}$ since $\E_x(D_\infty)=1$ for all $x \in J$ by Theorem \ref{theo:main}, if we show that $G$ has at most two fixed points then this would imply that $\eta \equiv \sigma$ and so \eqref{eq:eq} would follow at once. Unfortunately, it is not always the case that $G$ has only two fixed points, see Section \ref{sec:tou} for instance. Hence, we must impose some additional conditions for this to occur. First, we make some further assumptions.

\begin{assump}\label{assumpG2} Throughout Subsection \ref{sec:strict} we will make the following additional assumptions:
	\begin{enumerate}
	\item [B1.] $\Phi_x < +\infty$ for all $x \in J$.
	\item [B2.] For any $B \in \B_J$ with $\nu(B)>0$ there exists $B^* \in \mathcal{C}_X$ such that $B^* \subseteq B$ and $\nu(B^*)>0$.
	\item [B3.] The conditioned evolution of $X$ is \textit{irreducible}, i.e. for any pair $x \neq x' \in J$ and $B \in \B_J$ there exists $n=n(x,x',B) \in \N$ such that 
	$$
	 P_{x'}(X_1 \in B)> 0 \Longrightarrow P_{x}(X_{n+1} \in B) > 0.
	$$
	\end{enumerate}
\end{assump}

Assumption (B1) is not really restrictive, as we wish to focus here only on the case in which there is convergence in $L^2$. On the other hand, we impose (B2) in order to obtain an appropriate control on the growth of $\xi_t(B)$, namely that for each $n \in \N$ and $x \in J$
$$
\lim_{t \rightarrow +\infty} \left[\inf_{y \in J_n}\E_x(\xi_t(B)) \right]=+\infty \hspace{1cm}\text{ and }\hspace{1cm}\lim_{t \rightarrow +\infty} \xi^{(x)}_t(B) = +\infty \hspace{0.2cm}\text{ on }\{D^{(x)}_\infty > 0\},
$$ which follows from (B1-B2) by Theorem \ref{theo:main} (see Section \ref{sec:theo3} below). Furthermore, typically (B2) is very easy to check, see \cite[Section 9]{JonckSag} for details. Finally, the notion of irreducibility in (B3) is different than that of $\psi$-irreducibility featured in \cite{meyn1993} and weaker than the standard definition of irreducibility when $J$ is countable. Although not entirely standard, it is nevertheless the notion which appears naturally in our analysis and it is satisfied in all applications of interest, see \cite{JonckSag}.

Next, we introduce the notion of strong supercriticality which plays a key role in what follows.


\begin{definition} \label{def:ss} We shall say that the branching dynamics $\xi^{(x)}=(\xi^{(x)}_t)_{t \geq 0}$ is \textit{strongly supercritical} if:
	\begin{enumerate}
		\item [i.] $\xi^{(x)}$ is supercritical, i.e. $P_x(\Theta) > 0$.
		\item [ii.] $\eta(x)=P_x(\Gamma)$, where $\Gamma^{(x)}$ is the event defined as 
		$$
		\Gamma^{(x)}:=\left\{ \lim_{t \rightarrow +\infty} \left[ \min_{u \in \overline{\xi}_t^{(x)}} \Phi_{u_t}\right]= +\infty\right\}
		$$ with the convention that $\min_{u \in \emptyset} \Phi_{u_t}:=+\infty$, used whenever $|\xi_t|=0$. 
    \end{enumerate} Note that, provided (i) holds, (ii) is equivalent to the condition $P_x(\Gamma|\Theta)=0$. 
\end{definition}

One can check (see once again \cite[Section 9]{JonckSag}) that in all the examples of Section \ref{sec:examples} the mapping $x \mapsto \Phi_x$ is bounded over subsets of $J$ which are at a positive distance from $\partial_* \overline{J}$ and, on the other hand, that it tends to infinity as $x$ approaches $\partial_* \overline{J}$. Thus, one can interpret strong supercriticality as the condition stating that on the event of non-extinction particles never accumulate all together on the boundary of the state space, $\partial_* \overline{J}$.  
On the other hand, we will see later in Section \ref{sec:theo3} that under Assumptions \ref{assumpG2} strong supercriticality is equivalent to having 
\begin{equation} \label{eq:sscom}
P_x\left( \limsup_{t \rightarrow +\infty} \xi_t(B) > 0\right)=P_x(\Theta) > 0
\end{equation} for every $x \in J$ and all $B \in \B$ with $\nu(B) > 0$, which is the analogue in our context of the notion of strong local survival studied in \cite{bertacchi2014} and other references therein. However, in general it will not be equivalent to the concept of (plain) local survival introduced in \cite{EngKyp2004}, which is said to take place whenever there exists a compact set $\mathcal{K} \subseteq J$ such that
		\begin{equation}\label{eq:ls}
		P_x\left( \limsup_{t \rightarrow +\infty} \xi_t(\mathcal{K}) > 0\right) > 0.
		\end{equation} See Section \ref{sec:tou} for further details.
		
Our next result states that strong supercriticality is a necessary and sufficient condition for $G$ to have exactly two fixed points whenever under Assumptions \ref{assumpG2}.

\begin{theorem}\label{theo:main3} If Assumptions \ref{assumpG2} also hold then the following statements are equivalent:
\begin{enumerate}
	\item [i.] $G$ has exactly two fixed points, $\eta$ and $\mathbf{1}$.
	\item [ii.] $\eta(x) = \sigma(x)$ for all $x \in J$.
	\item [iii.] $\eta(x) = \sigma(x)$ for some $x \in J$.
    \item [iv.] $\xi^{(x)}$ is strongly supercritical for some $x \in J$. 
    \item [v.] $\xi^{(x)}$ is strongly supercritical for all $x \in  J$. 
\end{enumerate} 
\end{theorem}

We note that strong supercriticality is not a trivial condition under our current assumptions, not even for the particular case of $\lambda$-positive systems to be considered in Section \ref{sec:rpos1} below. Indeed, Section \ref{sec:tou} shows an example of a $\lambda$-positive system which satisfies Assumptions \ref{assumpG}, \ref{assumpG0} and \ref{assumpG2} but is not strongly supercritical. In particular, we have in this example that the random variable $D^{(x)}_\infty$ is zero with positive probability on the event $\Theta^{(x)}$ of non-extinction. 
Nonetheless, whenever $\xi^{(x)}$ is strongly supercritical this is not the case and so one obtains the following corollary.

\begin{corollary}\label{cor:convnu} If Assumptions \ref{assumpG2} hold and $\xi^{(x)}$ is strongly supercritical then $D^{(x)}_\infty > 0$ on $\Theta^{(x)}$. 
In particular, for every $B,B' \in \mathcal{C}_X$ with $\nu(B')> 0$ we have that, conditionally on $\Theta^{(x)}$, as $t \rightarrow +\infty$
	\begin{equation*} 
	\nu_t^{(x)}(B,B') \overset{P}{\longrightarrow} \frac{\nu(B)}{\nu(B')}.
	\end{equation*} 
\end{corollary}

Still, strong supercriticality appears to be a hard condition to check directly, at least in principle. In the extended version of this article \cite{JonckSag}, we introduce via examples some general methods to establish strong supercriticality which apply to a wide range of systems.

\subsection{The case of $\lambda$-positive systems}\label{sec:rpos1}

Perhaps the simplest example of an underlying motion satisfying Assumptions \ref{assumpG} is that of a $\lambda$-positive process, which we formally introduce now.

\begin{definition}
	\label{def:rpos0}
	Given $\lambda \in \R$ and a Markov semigroup $S=(S_t)_{t \geq 0}$ with \mbox{associated generator $\mathcal{L}_S$,} we will say that $S$ is $\lambda$-positive if there exist a nonnegative measurable function $h: \overline{J} \rightarrow [0,+\infty)$ satisfying $h|_J > 0$ and a (not necessarily finite) nonnegative measure $\nu$ on $(\overline{J},\B)$, both unique up to constant multiples, such that:
	\begin{enumerate}
		\item [$\bullet$] $S_t[h]=e^{-\lambda t}h$ for all $t$, i.e. $-\lambda$ is a right-eigenvalue of $\mathcal{L}_S$ with \mbox{associated eigenfunction $h$.}
		\item [$\bullet$] For any nonnegative measurable $f : \overline{J} \rightarrow \R_{\geq 0}$ and all $t$,
		$$
		\int_{\overline{J}}S_t[f](x)d\nu(x)=e^{-\lambda t}\int_{\overline{J}}f(x)d\nu(x),
		$$ i.e. $-\lambda$ is a left-eigenvalue of $\mathcal{L}_S$ with associated eigenmeasure $\nu$.
		\item [$\bullet$] The eigenvectors $h$ and $\nu$ are such that
		$$
		\nu(h):=\int_{\overline{J}} h(x) d\nu(x) < +\infty.
		$$ 		
	\end{enumerate}
In particular, we shall say that our branching dynamics $\xi$ is $\lambda$-positive whenever its associated expectation semigroup $S^{(\xi)}=(S_t^{(\xi)})_{t \geq 0}$ is $\lambda$-positive according to the definition given above, where for each $t \geq 0$ and nonnegative measurable $f:\overline{J} \rightarrow \R_{\geq 0}$ we define $$
S^{(\xi)}_t[f](x):=\E_x\left( \sum_{u \in N_t} f(u_t)\right).
$$
\end{definition}

\begin{remark} Using the many-to-one lemma (Lemma \ref{lema:mt1} below) it is straightforward to see that, in the current case of a constant branching rate $r>0$ (and only in this case), $\xi$ will be $\lambda$-positive if and only if the underlying motion $X$ is $(r(m_1-1)-\lambda)$-positive. For this reason, in the following we shall focus only on $\lambda$-positivity of underlying motions rather than that of branching dynamics, as it makes no difference in our current setting.
\end{remark}

Note that if $S$ is a $\lambda$-positive Markov semigroup then its right-eigenfunction $h$ satisfies $h \equiv 0$ on $\partial_* J$. With this, we may define a Markov semigroup $\tilde{S}=(\tilde{S}_t)_{t \geq 0}$ on the space $J$ by the formula
	$$
	\tilde{S}_t[f](x)= \frac{e^{\lambda t}}{h(x)}S_t[hf](x)
	$$ 
	for any nonnegative $\B_J$-measurable $f:J \rightarrow \R_{\geq 0}$, $x \in J$ and $t \geq 0$, where we set $hf \equiv 0$ on $\partial_* J$. We will call $\tilde{S}$ the $h$-\textit{transform} of $S$. Let us notice that, in the case $S$ is the semigroup belonging to a $\lambda$-positive underlying motion $X$, $\tilde{S}$ is none other than the semigroup corresponding to $X$ under the $h$-transform defined in \eqref{eq:mcof}. It is straightforward to see that, if one normalizes $h$ and $\nu$ so that $\nu(h)=1$, then the probability distribution $\mu$ on $(J,\B_J)$ defined through the formula
	$$
	\frac{d\mu}{d\nu}=h
$$ is invariant for $\tilde{S}$, i.e. for all nonnegative $\B_J$-measurable $f:J \rightarrow \R_{\geq 0}$ and all $t \geq 0$
$$
\int_J \tilde{S}_t[f](x)d\mu(x) = \int_J f(x)d\mu(x).
$$ Let us assume further that $X$ is ergodic under the $h$-transform, with limiting distribution $\mu$. Then, for any $B \in \B_J$ such that the function $\frac{\mathbbm{1}_B}{h}$ is bounded and $\mu$-almost surely continuous, we have 
$$
\lim_{t \rightarrow +\infty} \tilde{\E}_x\left( \frac{\mathbbm{1}_B}{h}(X_t)\right) = \mu\left( \frac{\mathbbm{1}_B}{h}\right)= \nu(B)
$$ as $t \rightarrow +\infty$, so that by \eqref{eq:fasymp} the asymptotic formula \eqref{A2} holds for any such $B$ by taking $p(t)\equiv 1$. 
Hence, we see that such $\lambda$-positive motions fall naturally into the context of Assumptions \ref{assumpG}. \mbox{However,} neither the uniform convergence of $s_B$ over each $J_n$ nor the square-integrability of \mbox{$M_t$ will} follow immediately from the ergodicity of $X$ under the $h$-transform, so that further conditions will need to be imposed on the process to guarantee them. This is not a disadvantage when compared to other approaches in the literature, as additional conditions are always imposed in order to obtain a \mbox{law of large numbers,} see \cite{chen2016,englander2010}. Here, we propose an alternative condition to check the remainder of Assumptions \ref{assumpG} in the $\lambda$-positive setting and obtain Theorem \ref{theo:main}, based on the existence of a Lypaunov functional for the process $X$ under the measure $\tilde{P}$.

\begin{definition}\label{def:lyapunov} A $\B_J$-measurable $V : J \rightarrow \R_{\geq 0}$ is called a \mbox{(\textit{geometric}) \textit{Lyapunov functional for }$X$} (under the measure $\tilde{P}$) whenever it satisfies: 
\begin{itemize}
	\item [V1.] There exists $t > 0$ such that for every $R > 0$ one can find $\alpha_R \in (0,1)$ verifying  
	$$
	|\tilde{\E}_x(f(X_t))-\tilde{\E}_y(f(X_t))| \leq 2(1-\alpha_R)\|f\|_\infty
	$$ for any bounded $\B_J$-measurable $f:J \rightarrow \R$ and all $x,y \in J$ such that $V(x)+V(y) \leq R$.
	\item [V2.] There exist constants $\gamma,K > 0$ such that for all $t \geq 0$ and $x \in J$ one has
	$$
	\tilde{\E}_x(V(X_t)) \leq e^{-\gamma t}V(x) + K.
	$$
\end{itemize} 
\end{definition}

Having a Lyapunov functional ensures that, under the $h$-transform $\tilde{P}$, the process $X$ converges to equilibrium exponentially fast and, furthermore, that it does so uniformly over subsets of $J$ where $V$ is bounded, see Proposition \ref{prop:Lyapunov} below. As a consequence, 
one has the following result relating the validity of Theorem \ref{theo:main} for $\lambda$-positive processes to the existence of a $h$-locally bounded Lyapunov functional for $X$ with a large enough growth at infinity. 

\begin{proposition}\label{prop:lyapunov1}
	If $X$ is $\lambda$-positive and admits a Lyapunov functional $V$ such that:
	\begin{enumerate}
		\item [V3.] $V$ is $h$-locally bounded, i.e. $\sup_{x \in J_n} V(x) <+\infty$ for each $n \in \N$,
		\item [V4.] $\left\|\frac{h}{1+V}\right\|_\infty <+\infty$,
\end{enumerate}
    then Assumptions \ref{assumpG} are satisfied for all $B \in \mathcal{C}_X$, where $\mathcal{C}_X$ here is given by 
	\begin{equation}
	\label{eq:defcx}
	\mathcal{C}_X:= \left\{ B \in \B_J :  \left\| \frac{\mathbbm{1}_B}{h}\right\|_\infty <+\infty \right\}.
	\end{equation} Furthermore, there exists a constant $C=C(r(m_1-1),\lambda,V) > 0$ such that for all $x \in J$
	\begin{equation} \label{eq:compvc}
	\Phi_x \leq C \cdot \frac{1+V(x)}{h(x)} < +\infty
	\end{equation} In particular, the convergences in \eqref{eq:conv1} and \eqref{eq:ks2} hold for any $x \in J$ and $B,B' \in \mathcal{C}_X$ with $\nu(B')> 0$, and assumptions (B1)-(B2) are also satisfied.
\end{proposition} 


Finally, whenever $X$ is $\lambda$-positive and admits such a Lyapunov functional, as a matter of fact one can show \eqref{eq:conv1}-\eqref{eq:ks2} in the almost sure sense provided that, on the other hand, $V$ does not grow too fast at infinity. This is the content of our last result.

\begin{theorem}\label{theo:main5} Suppose that $\xi$ is a $\lambda$-positive process such that the eigenfunction $h$ is continuous and $X$ admits a Lyapunov functional $V$ verifying (V3-V4). Then, for each $x \in J$ such that 
	\begin{equation}
	\label{eq:phibarra}
	\overline{\Phi}_x:=\frac{(m_2-m_1)r}{h(x)} \int_0^\infty \tilde{\E}_x\left(h(X_s)(1+V(X_s))\right)e^{-(r(m_1-1)-\lambda)s}ds < +\infty
	\end{equation} there exists a full $P$-measure set $\Omega^{(x)}$ satisfying that:
	\begin{enumerate}
		\item [i.] For any $\omega \in \Omega^{(x)}$ one has
	\begin{equation}\label{eq:convas}
	\lim_{t \rightarrow +\infty} \frac{\xi^{(x)}_t(B)(\omega)}{\E_x(\xi_t(B'))}= \frac{\nu(B)}{\nu(B')}\cdot D^{(x)}_\infty (\omega)
	\end{equation} for all pairs $B,B' \in \mathcal{C}_X$ satisfying $\nu(\partial B)=0$ and $\nu(B')> 0$, where $\mathcal{C}_X$ is given by \eqref{eq:defcx}. 
	\item [ii.] For any $\omega \in \Omega^{(x)} \cap \Lambda^{(x)}$ one has
	\begin{equation} \label{eq:convas2}
	\lim_{t \rightarrow +\infty} \frac{\xi^{(x)}_t(B)(\omega)}{\xi^{(x)}_t(B')(\omega)} = \frac{\nu(B)}{\nu(B')}
	\end{equation} for all pairs $B,B' \in \mathcal{C}_X$ with $\nu(\partial B) = \nu(\partial B')=0$ and $\nu(B')> 0$. 
	\end{enumerate} 
\end{theorem} 

The proof of Theorem \ref{theo:main5} goes much along the lines of that of similar results found in \cite{englander2010, ChenShio2007}, so again we shall choose not to include it here and refer to \cite[Section 8]{JonckSag} for complete details. Finally, we note that Theorem \ref{theo:main5} can be combined with Theorem \ref{theo:main3} in order to obtain \eqref{eq:convas2} for any $\omega \in \Omega^{(x)} \cap \Theta^{(x)}$ whenever $\xi^{(x)}$ is strongly supercritical.

\subsection{Sketches of the proofs}

We conclude this section by discussing the main ideas in the proofs for each of our results.

The proof of Theorem \ref{theo:main2} essentially boils down to explicitly computing $\E_x(D_t^2)$ for each $x \in J$ by means of the many-to-two lemma (Lemma \ref{lema:mt2} below). Using the fact that the process $(M^{(x)}_t)_{t \geq 0}$ is a mean-one martingale for all $x \in J$, we will be able to show that for all $t \geq 0$
$$
\E_x(D_t^2)=\E_x(M_t^2)e^{-r(m_1-1)t} + (m_2-m_1)r\int_0^t \E_x(M_s^2)e^{-r(m_1-1)s}ds
$$ from where a simple analysis will then yield the result. 

As for the proof of Theorem \ref{theo:main}, the main step is to show that for any $x \in J$ and pair $B,B' \in \mathcal{C}_X$ with $\nu(B')\neq 0$ one has 
\begin{equation} \label{eq:proof1}
\lim_{t \rightarrow +\infty} \E_x(W^2_t(B,B')) = \left[\frac{\nu(B)}{\nu(B')}\right]\Phi_x.
\end{equation} To see this, we will use the many-to-few lemmas together with the fact that the function $p$ in (A2) has subexponential growth (see Lemma \ref{lema:p} below) to show that
\begin{equation} \label{eq:proof2}
\E_x(W^2_t(B,B')) \approx (m_2-m_1)r \int_0^t \frac{\E_x(P^2_{X_s}(X_{t-s} \in B))}{P_x^2(X_t \in B')} e^{-r(m_1-1)s}ds,
\end{equation} where the notation $a(t) \approx b(t)$ here means that $a(t)-b(t) \rightarrow 0$ as $t \rightarrow +\infty$. On the other hand, the asymptotic formula in (A2) allows us to write for each $s \in [0,t]$
$$
\frac{\E_x(P^2_{X_s}(X_{t-s} \in B))}{P_x^2(X_t \in B')}
\sim \left[\frac{p(t-s)}{p(t)}\right]^2 \frac{1}{\nu^2(B')} \E_x(M_s^2 (\nu(B)+s_B(X_s,t-s))^2),
$$where $a(t) \sim b(t)$ now means that $\frac{a(t)}{b(t)} \rightarrow 1$ as $t$ tends to infinity. Thus, the limit \eqref{eq:proof1} will follow once we use our assumptions to show that:
\begin{enumerate}
	\item [i.] The behavior of the integral in the right-hand side of \eqref{eq:proof2} is dictated by the bulk and not by its tail, i.e. only the integral until some time $T \ll t$ is relevant.
	\item [ii.] For $s \leq T$ we have 
	$$
	\left[\frac{p(t-s)}{p(t)}\right]^2 \sim 1 \hspace{1cm}\text{ and }\hspace{1cm}\E_x(M_s^2 (\nu(B)+s_B(X_s,t-s))^2) \sim \nu^2(B)\E_x(M_s^2).
	$$ 
\end{enumerate} 
Now, having shown \eqref{eq:proof1}, it is then clear that $W_t^{(x)}(B,B')$ will not converge in $L^2$ if $\Phi_x=+\infty$. On the other hand, since 
$$
\left\|W^{(x)}_t(B,B') - \frac{\nu(B)}{\nu(B')}D_t^{(x)}\right\|_{L^2} = \E_x(W_t^2(B,B')) -2\left[\frac{\nu(B)}{\nu(B')}\right]\E_x(W_t(B,B')D_t) + \left[\frac{\nu(B)}{\nu(B')}\right]^2\E_x(D_t),
$$ if $\Phi_x < +\infty$ then, by \eqref{eq:proof1} and Theorem \ref{theo:main2}, the $L^2$-convergence will follow once we show that
$$
\lim_{t \rightarrow +\infty} \E_x(W_t(B,B')D_t) = \left[\frac{\nu(B)}{\nu(B')}\right] \Phi_x,
$$ which can be done in the same manner as for \eqref{eq:proof1}. To conclude, \eqref{eq:ks2} then follows from elementary properties of convergence in probability.

On the other hand, the proof of Theorem \ref{theo:main3} consists of three steps:
\begin{enumerate}
	\item [i.] Prove that if $h \neq \eta,\mathbf{1}$ is a fixed point of $G$ then $\eta(x) < h(x)< 1$ for all $x \in J$.
	\item [ii.] Show that the process is strongly supercritical if and only if conditionally on non-extinction every set of the form $\tilde{J}_n:=\{y \in J : \Phi_y \leq n\}$ for $n \in \N$ is visited infinitely many times, and these times go all the way up to infinity, i.e. $\limsup_{t \rightarrow +\infty} \xi^{(x)}(\tilde{J}_n) > 0$.
	\item [iii.] In the event that $\limsup_{t \rightarrow +\infty} \xi^{(x)}(\tilde{J}_m)>0$ for some $n \in \N$, the amount of particles inside any given set $B \in \mathcal{B}_X$ with $\nu(B)>0$ diverges as $t \rightarrow +\infty$. 
\end{enumerate} Now, if $h$ is any fixed point of $G$ aside from $\mathbf{1}$, by elementary properties of $G$ and definition of $\eta$ one can show that for any $\varepsilon > 0$
\begin{align}
h(x) &= \liminf_{n \rightarrow +\infty} \E_x\left( \prod_{u \in \overline{\xi}_n} h(u_n)\right)\nonumber\\ 
& \leq \liminf_{n \rightarrow +\infty} \left(P_x(N_n=0) + \E_x\left( (1-\varepsilon)^{\xi_n(B_\varepsilon)} \mathbbm{1}_{\{N_n \neq 0\}}\right)\right)\nonumber\\ 
& = \eta(x) + \liminf_{n \rightarrow +\infty} \E_x\left( (1-\varepsilon)^{\xi_n(B_\varepsilon)} \label{eq:proof3} \mathbbm{1}_{\Theta}\right),
\end{align} where $B_\varepsilon := \{x \in J : h(x)<1-\varepsilon\}$. By (i) we know there exists $\varepsilon > 0$ sufficiently small such that $\nu(B_\varepsilon)>0$ so that, by (ii) and (iii), one can conclude that the rightmost term in \eqref{eq:proof3} is zero and thus $h(x)\leq \eta(x)$ which, by (i) again, implies that $h=\eta$. This argument shows that if $\xi^{(x)}$ is strongly supercritical for some $x \in J$ then $G$ has exactly two fixed points, $\eta$ and $\mathbbm{1}$. The reverse implication is obtained from (ii) by reformulating strong supercriticality in terms of equalities between certain fixed points of $G$. The other implications in the statement of Theorem \ref{theo:main3} are then immediate.  

The proof of (i) is a standard argument using the irreducibility of $\xi$, while the proof of (ii)-(iii) relies on a coupling argument dominating $\xi$ from below on the event $\{\limsup_{t \rightarrow +\infty} \xi^{(x)}(\tilde{J}_m)>0\}$ by a suitable sequence of independent single-type supercritical Galton-Watson processes, for which we already know that survival implies divergence of the number of individuals.

Finally, Proposition \ref{prop:lyapunov1} follows from the fact (which is shown in Proposition \ref{prop:Lyapunov} below) that if $V$ is a Lyapunov functional for $X$ then $\mu(g) < +\infty$ for any $\B_J$-measurable $g:J \rightarrow \R$ such that $\|\frac{g}{1+V}\|_\infty < +\infty$ and, furthermore, there exist constants $C,\gamma > 0$ such that for all $x \in J$, $t \geq 0$ and any $g$ as above one has
\begin{equation} \label{eq:V}
|\tilde{\E}_x(g(X_t)) - \mu(g)| \leq C (1+V(x)) e^{-\gamma t} \left\|\frac{g-\mu( g)}{1+V}\right\|_\infty.
\end{equation} We use \eqref{eq:V} with $g=h$ to show that $\E(M_t^2)=\tilde{\E}_x(M_t) < +\infty$ for all $t \geq 0$ and then with $g=\frac{\mathbbm{1}_B}{h}$ for any $B$ as in \eqref{eq:defcx} to obtain (A2) from the formula for $s_B$ in \eqref{eq:defsb} but with $p(t)\equiv 1$ (which holds since $X$ is $\lambda$-positive). This concludes the proof.

\section{Examples and applications} \label{sec:examples}

We now illustrate our results through a series of examples. We present first the case of generic ergodic motions, which fall in the category of $0$-positive processes, and then proceed on to study four different models with $\lambda$-positive motions for $\lambda > 0$. Finally, we conclude in Section \ref{sec:bbm} with our results for the Branching Brownian Motion with a negative drift and absorption at the origin, which is a canonical example of a system with an underlying motion which is not $\lambda$-positive and constitutes our most significant and novel contribution. Further details as well as the verification that all of our required assumptions are met in each of the examples can be found in the extended version of the article, \cite[Section 9]{JonckSag}.  

\subsection{Ergodic motions}\label{sec:cwa}

Suppose that $X$ is a motion without absorbing states, i.e. $\overline{J}=J$, which is ergodic and has stationary probability distribution $\nu$ on $(J,\B_J)$. In this case, it is easy to check that $X$ is $0$-positive and verifies \eqref{A2} with $h \equiv 1$, $p \equiv 1$, $\lambda = 0$ for any $B \in \B_J$ such that 
$$
s_B(x,t):= P_x(X_t \in B) - \nu(B) \underset{t \rightarrow +\infty}{\longrightarrow} 0.
$$ One can then show that Theorems \ref{theo:main2}-\ref{theo:main} hold with $\mathcal{C}_X$ given by 
\begin{equation}
\label{eq:c}
\mathcal{C}_X:=\left\{ B \in \B_J : \lim_{t \rightarrow +\infty} s_B(\cdot,t) = 0 \text{ uniformly over compact sets of $J$}\right\}.
\end{equation} In this case, each random variable $D^{(x)}_\infty$ satisfies 
$$
\E_x(D_\infty)=1 \hspace{2cm}\text{ and }\hspace{2cm}\E_x(D_\infty^2)=:\Phi_x=\frac{m_2-m_1}{m_1-1}.
$$ Since there is no absorption here and $\sup_{x \in J} \Phi_x < +\infty$, the corresponding branching dynamics is immediately strongly supercritical. Thus, if in addition $X$ is irreducible in the sense of (B3) then Corollary \ref{cor:convnu} holds and $D^{(x)}_\infty > 0$ on $\Theta^{(x)}$. Finally, if $X$ admits a Lyapunov functional (as defined in Definition \ref{def:lyapunov}) which is bounded over compact subsets of $J$ then the contents of Theorem \ref{theo:main5} can also be shown to hold. See \cite[Section 3.1]{JonckSag} for details.

\subsection{Subcritical Galton-Watson process}\label{sec:gw} Let us now consider $\lambda$-positive motions with $\lambda > 0$.
As a first example, let $X$ be a continuous-time Galton-Watson process, i.e. a process on $\overline{J}:=\N_0$ with transition rates $q$ given for any $x \in \N_0$ and $y \in \N^*:=\{-1\} \cup \N_0$ by
$$
q(x,x+y):=x \rho(y), 
$$ where $\rho$ is some probability vector in $\N^*$ representing the offspring distribution of each individual in the branching process (minus $1$). We assume that $X$ is subcritical, i.e. that $-\lambda:= \sum_y y \rho(y) <0$, so that $X$ is almost surely absorbed at $0$. 
It is well-known, see \cite{seneta2006}, that in this case $X$ is $\lambda$-positive with associated eigenfunction $h$ given by $h(x) \propto x$ and that $\nu$ is a finite measure assigning positive mass to all $x \in \N$, although an explicit expression for $\nu$ is, in general, \mbox{not known.} 

If $\rho$ has a finite second moment then $D^{(x)}$ is bounded in $L^2$ and thus the contents of Theorems \ref{theo:main2}-\ref{theo:main} all hold in this case. Furthermore, using $h$ as a Lyapunov functional for $X$, one can exploit \eqref{eq:compvc} to verify that $\xi$ is strongly supercritical.  Thus, if we $\rho(-1) \in (0,1)$ then $X$ is irreducible (in the classical sense) and, as a consequence, we obtain Theorem \ref{theo:main3} and Corollary \ref{cor:convnu}. Finally, if $\rho$ has a finite \textit{third} moment then \eqref{eq:phibarra} is satisfied and thus the contents of Theorem \ref{theo:main5} also hold.

\subsection{Subcritical contact process on $\Z^d$ (modulo translations)}\label{sec:cp}

Let $\mathcal{P}_{f}(\Z^d)$ denote the class of all finite subsets of $\Z^d$ and $Y=(Y_t)_{t \geq 0}$ be the contact process on $\Z^d$, i.e. the Markov process on $\mathcal{P}_{f}(\Z^d)$ with transition rates $q$ given, for any $\sigma \in \mathcal{P}_{f}(\Z^d)$ and $x \in \Z^d$, by 
$$
q(\sigma,\sigma \cup \{x\}) = \gamma |\{ y \in \sigma : |y-x|_1=1\}| \hspace{1cm} \text{ and }\hspace{1cm} q(\sigma,\sigma-\{x\}) = \mathbbm{1}_{\sigma}(x),
$$ where $\gamma > 0$ is a fixed constant called the \textit{infection rate}. Notice that $Y$ is translation invariant, i.e. for any $\sigma \in \mathcal{P}_f(\Z^d)$ and $x \in \Z^d$ one has
$$ 
Y^{(\sigma)} \sim Y^{(\sigma + \{x\})} + \{-x\},
$$ so that there are no finite measures $\nu \neq 0$ verifying \eqref{A2} for $Y$.  
This can be fixed if one considers the process modulo translations. Indeed, say that two non-empty sets $\sigma,\sigma' \in \mathcal{P}_f(\Z^d)$ are \textit{equivalent} if they are translations of each other. Let $J$ denote the quotient space obtained from this equivalence and, for any non-empty $\sigma \in \mathcal{P}_f(\Z^d)$, let $\langle \sigma \rangle$ denote its corresponding equivalence class in $J$. Also, set $\langle \emptyset \rangle := \emptyset$ and $\overline{J}:= J \cup \{ \emptyset\}.$ Then, for any $\zeta \in J$ we define $X^{(\zeta)}$ by taking $\sigma_\zeta \in \mathcal{P}_f(\Z^d)$ such that $\langle \sigma_\zeta \rangle = \zeta$ and setting 
$
X^{(\zeta)}_t:= \langle  Y^{(\sigma_\zeta)}_t \rangle$. We call $X$ the \textit{contact process on $\Z^d$ modulo translations}.

It is well-known, see \cite{BezGri1990}, that $J$ is an irreducible class for the process $X$ and that there exists $\gamma_c=\gamma_c(d) > 0$ such that the absorbing state $\emptyset$ is reached almost surely if and only if $\gamma \leq \gamma_c$. Moreover, it has been shown that for $\gamma < \gamma_c$ then $\lambda, h$ and $\nu$ as in Assumptions \ref{assumpG} indeed exist and that the process is in fact $\lambda$-positive, see \cite{ferrari1996,andjel2015}, although neither $h$ nor $\nu$ are explicitly known. What is known, however, is that $\nu$ is finite and it assigns positive mass to every $x \in J$, see \cite{ferrari1996}. Using $h$ again as a Lyapunov functional, one can show that for any subcritical infection rate $\gamma < \gamma_c$ the Malthusian martingale $D^{(\zeta)}$ is bounded in $L^2$ for all $\zeta$ and that the branching dynamics $\xi^{(\zeta)}$ is strongly supercritical, so that the contents of all our main results always hold for this model.

\subsection{Recurrent Ornstein-Uhlenbeck process killed at $0$} \label{sec:rou}

Consider a $1$-dimensional recurrent Ornstein-Ulhenbeck process which is killed at $0$, i.e. the stopped process $X=(X_t)_{t \geq 0}$ on $\R_{\geq 0}$ defined as $X_t:=Y_{t \wedge \tau_0}$, where $\tau_0:=\inf\{t \geq 0 : Y_t = 0\}$ and $Y$ is given by SDE 
$$
dY_t = -\lambda Y_t dt + dB_t, 
$$ for $B$ a standard ($1$-dimensional) Brownian motion and $\lambda > 0$ a fixed parameter called the \textit{drift}. 
The generator of $X$ has as domain the set of $C^2$-functions vanishing at $0$ (due to the killing at $0$) and is defined for any such $f$ as 
$$
\mathcal{L}[f](x):= \frac{1}{2}f''(x) - \lambda xf'(x).
$$ 
It is well-known, see \cite{lladser2000}, that $X$ is $\lambda$-positive with eigenfunction $h(x):= \sqrt{\frac{4\lambda}{\pi}} x$ (when $h,\nu$ are normalized so that $\nu$ is a probability measure and $\nu(h)=1$) and eigenmeasure $\nu$ having density 
$
f_X(x):= 2\lambda xe^{-\lambda x^2} \mathbbm{1}_{{(0,+\infty)}}(x)
$ with respect to the Lebesgue measure $l$ on $\R$. As in the previous example, the Malthusian martingale $D^{(x)}$ is bounded in $L^2$ for all $x$ and the branching dynamics $\xi^{(x)}$ is strongly supercritical, so that the contents of all our main results always hold for this model.

\subsection{Transient Ornstein-Uhlenbeck process} \label{sec:tou}
This example was considered originally in \cite{englander2010}. Let $X$ be the process with generator $\mathcal{L}$ defined for any $f \in C^2(\R)$ as 
$$
\mathcal{L}[f](x):= \frac{1}{2}\sigma^2 f''(x) +\lambda xf'(x),
$$ 
where $\lambda,\sigma^2 > 0$ are called the \textit{drift} and \textit{dispersion} coefficients, respectively. In this case, one can check that $X$ is $\lambda$-positive with $h(x) \propto \exp\{-\frac{\lambda}{\sigma^2}x^2\}$ and $\nu$ given by the Lebesgue measure on $\R$. Unlike previous examples, here $\nu$ is an infinite measure so that, in particular, the asymptotics in \eqref{A2} does not hold for every $B \in \B_{\R}$ (for example, it does not hold if $B=\R$ because in this case \eqref{A2} cannot decay exponentially) but, by the transience of $X$, it does hold for any $B$ which is bounded. Still, this is enough to yield the following results:

\begin{itemize}
		\item [i.]  $D^{(x)}$ converges almost surely and in $L^2$ to some random variable $D^{(x)}_\infty \in L^2$ for each $x \in \R$.
		\item [ii.] The contents of Theorem \ref{theo:main} hold with $
		\mathcal{C}_X$ the class of \textit{bounded} Borel subsets of $\R$.
		\item [iii.] $\xi^{(x)}$ is \textbf{not} strongly supercritical for any $x$. In fact, $0 < \sigma(x) < 1$ for all $x$ so that $\sigma \neq \eta,\mathbf{1}$.
		\item [iv.] $\overline{\Phi}_x < +\infty$ for all $x$, so that the contents of Theorem \ref{theo:main5} also hold in this case.
	\end{itemize}
Parts (i) and (iv) of Theorem \ref{theo:rou} can also be found in \cite{englander2010} (together with the $L^1$ convergence of $W^{(x)}(B,B')$), whereas (ii) and (iii) are new results. On the other hand, we note that it follows from (iv) that for any compact set $\mathcal{K} \subseteq \R$ and $x \in \R$ we have 
$$
P_x\left( \limsup_{t \rightarrow +\infty} \xi_t(\mathcal{K}) > 0\right) \geq P_x (\Lambda) > 0
$$ but, however, from (iii) that $\xi^{(x)}$ is not strongly supercritical. This confirms our statement about the notion of local survival introduced in \eqref{eq:ls} being, in general, weaker than strong supercriticality.

\subsection{Brownian motion with drift killed at the origin} \label{sec:bbm}

Finally, we conclude with an example of an underlying motion which is not $\lambda$-positive. Consider a Brownian motion with negative drift $-c < 0$ killed at the origin, i.e. the process $X$ given by the generator 
$$
\mathcal{L}[f](x):=\frac{1}{2}f''(x) - c f'(x)
$$ defined for all $C^2$-functions $f$ vanishing at $0$. It is shown in \cite{polak2012} that $X$ satisfies Assumptions \ref{assumpG} for $\lambda:=\frac{c^2}{2}$ and $\mathcal{C}_X:=\B_{(0,+\infty)}$, the class of Borel subsets of $(0,+\infty)$, and with $h(x):= \frac{1}{\sqrt{2\pi \lambda^2}} xe^{cx}$, $p(t):=t^{-\frac{3}{2}}$ and $\nu$ given by the density $f_X(x):=2\lambda x e^{-cx}\mathbbm{1}_{(0,+\infty)}(x)$ with respect to Lebesgue. However, $X$ is not a $\lambda$-positive motion since one can easily verify that $\nu(h)=+\infty$.\footnote{Nor is it $\lambda'$-positive for any other $\lambda'\neq \frac{c^2}{2}$ because otherwise \eqref{A2} would not hold for $\lambda=\frac{c^2}{2}$ and all $B \in \B_{(0,+\infty)}$.} Nonetheless, our approach still applies and can thus obtain the following:

\begin{itemize}
		\item [i.] For each $x > 0$ the martingale $D^{(x)}$ is bounded in $L^2$ if and only if $r(m_1-1) > c^2=2\lambda$ (note that this is strictly contained in the supercritical region $r(m_1-1)>\lambda$). In this case, the contents of Theorems \ref{theo:main2}-\ref{theo:main} all hold.
		\item [i.] $\xi^{(x)}$ is strongly supercritical for each $x$. In particular, the contents of Theorem \ref{theo:main3} and Corollary \ref{cor:convnu} hold.  
	\end{itemize}

\section{The many-to-few lemmas}\label{sec:mtf}

An element which will prove to be crucial in the proof of Theorem \ref{theo:main} is the ability to compute the first and second moments of the process $|\xi|=(|\xi_t|)_{t \geq 0}$ in exact form. We do this with the help of the many-to-few lemmas we state below. For simplicity, we will state only a reduced version of the many-to-one and many-to-two lemmas, which are all we need. For the many-to-few lemma in its full generality (and its proof) we refer to \cite{harris2015}. 

First, we state the many-to-one lemma. It receives this name because it reduces expectations involving random sums over \textit{many} particles, i.e. over all those in $\xi_t$, to expectations involving only \text{one} particle. 

\begin{lemma}[Many-to-one Lemma] \label{lema:mt1} Given a nonnegative measurable function $f:(\overline{J},\B) \rightarrow \R_{\geq 0}$, for every $t \geq 0$ and $x \in J$ we have
	$$
	\E_x \left( \sum_{u \in \overline{\chi}_t} f( u_t ) \right) = e^{r(m_1 - 1)t}\E_x \left( f(X_t) \right).
    $$
\end{lemma}

Next, we state the many-to-two lemma, used to compute correlations between pairs of particles. Before we can do so, however, we must introduce the notion of $2$-spine for our branching dynamics. 

\begin{definition} Consider the following coupled evolution on $\overline{J}$:
	\begin{enumerate}
		\item [i.] The dynamics starts with 2 particles, both located initially  at some $x \in J$, whose positions evolve together randomly, i.e. describing the same random trajectory, according to $\mathcal{L}$. 
		\item [ii.] The particles wait for an independent random exponential time $E$ of parameter $(m_2-m_1)r$ and then split at their current position, each of them then evolving independently afterwards according to $\mathcal{L}$. 
	\end{enumerate} 
	Now, for $i=1,2$, let $X^{(i)}=(X^{(i)}_t)_{t \geq 0}$ be the process which indicates the position of the $i$-th particle. We call the pair $(X^{(1)},X^{(2)})$ a $2$-spine associated to the triple $(m,r,\mathcal{L})$ and $E$ its splitting time. 
	\end{definition}
	
The many-to-two lemma then goes as follows.

\begin{lemma}[Many-to-two Lemma] \label{lema:mt2} Given any pair of measurable functions $f,g:(\overline{J},\B) \rightarrow \R_{\geq 0}$, for every $t \geq 0$ and $x \in J$ we have
		$$
		\E_x \left( \sum_{u,v \in \overline{\chi}_t} f( u_t )g(v_t) \right) = e^{2r(m_1 - 1)t}\E_x \left( e^{[\text{Var}(m) + (m_1-1)^2]r (E \wedge t)}f(X^{(1)}_t)g(X^{(2)}_t)\right),
		$$ where $(X^{(1)},X^{(2)})$ is a $2$-spine associated to $(m,r,\mathcal{L})$ and $E$ denotes its splitting time.
	\end{lemma}

\section{Proof of Theorem \ref{theo:main2}}\label{sec:proof2}

We first compute $\E_x( D_t^2)$ for every $t \geq 0$ and $x \in J$. Note that, by the many-to-two lemma and the definition of $2$-spine, a straightforward computation (see the proof of Theorem \ref{theo:main} for details) yields that
\begin{align*}
\E_x(D_t^2) &= \frac{1}{h^2(x)}\E_x \left( \sum_{u,v \in \overline{\xi}_t} h(u_t)h(v_t)e^{-2(r(m_1-1)-\lambda)t}\right)\\
&= \frac{e^{2\lambda t}}{h^2(x)}\E_x \left( e^{[\text{Var}(m) + (m_1-1)^2]r (E \wedge t)}h(X^{(1)}_t)h(X^{(2)}_t)\right)\\
& = [1]_t + [2]_t
\end{align*} where
$$
[1]_t := \frac{e^{2\lambda t}}{h^2(x)}\E_x \left( e^{[\text{Var}(m) + (m_1-1)^2]r (E \wedge t)}h^2(X^{(1)}_t)\mathbbm{1}_{\{E > t\}}\right) = \E_x(M_t^2)e^{-r(m_1-1)t}
$$
and 
$$
[2]_t := (m_2-m_1)r\int_0^t \E_x ( M_s^2\, \E_{X_s}^2 (M_{t-s}) ) e^{-r(m_1-1)s}ds. 
$$ Now, by (A1) we have that $M$ is a mean-one martingale so that $\E_{X_s}(M_{t-s})=1$ for all $s \in [0,t]$. Thus, we obtain that
$$
[2]_t= (m_2-m_1)r\int_0^t \E_x (M_s^2) e^{-r(m_1-1)s}ds.
$$ Recalling the definition of $\Phi_x$, it is then clear that $[2]_t \rightarrow \Phi_x$ as $t \rightarrow +\infty$ and, on the other hand, that whenever $\Phi_x < +\infty$ we have that $\liminf_{t \rightarrow +\infty}[1]_t \rightarrow 0$, so that $\liminf_{t \rightarrow +\infty} \E_x(D_t^2)= \Phi_x$. But, since $\lim_{t \rightarrow +\infty} \E_x(D_t^2)$ always exists (although it can be $+\infty$, in principle) because $\left(D^{(x)}\right)^2$ is a submartingale, we conclude that
$$
\lim_{t \rightarrow +\infty} \E_x( D_t^2) = \Phi_x.
$$ Being $D^{(x)}$ a martingale, this implies that it converges in $L^2$ if and only if $\Phi_x<+\infty$ and that, in this case, one has $\E_x(D_\infty^2)=\Phi_x$. Moreover, since $\E_x(D_t)=1$ for all $t \geq 0$, it also follows that $\E_x(D_\infty)=1$ and so this concludes the proof. 

\section{Proof of Theorem \ref{theo:main}}\label{sec:proof}

This section contains the proof of Theorem \ref{theo:main}. We will split the proof into two parts:

\begin{enumerate}
	\item [I.] First, we will show that, given $B,B' \in \mathcal{C}_X$ with $\nu(B')> 0$, for any $x \in J$ one has
	$$
	\lim_{t \rightarrow +\infty} \E_x(W^2_t(B,B')) = \left[\frac{\nu(B)}{\nu(B')}\right]\Phi_x.
    $$ 
	\item [II.] Then, we use (I) to conclude the convergence in \eqref{eq:conv1} whenever $\Phi_x < +\infty$. In particular, the convergence $W_t^{(x)}(B',B') \overset{P}{\longrightarrow} D_\infty^{(x)}$ together with \eqref{eq:conv1} yields that for any $B \in \mathcal{C}_X$ 
	$$
	\nu_t^{(x)}(B,B') \overset{P}{\longrightarrow} \frac{\nu(B)}{\nu(B')}
	$$ as $t \rightarrow +\infty$, conditionally on the event $\{D_\infty^{(x)} > 0\}$. 
\end{enumerate}
We dedicate a separate subsection to each parts, but begin first with a section devoted to proving two auxiliary lemmas to be used throughout the proof.

\subsection{Preliminary lemmas}

The first lemma we shall require is the following.

\begin{lemma}\label{lema:A3} If assumption (A1) holds then for any $T > 0$ we have
	\begin{equation} \label{eq:A3-ii}
	\lim_{n \rightarrow +\infty} \left[\sup_{t \in [0,T]} \E_x \left( M_t^2 \mathbbm{1}_{\{X_t \notin J_n\}}\right)\right] = 0.
	\end{equation}
\end{lemma}

\begin{proof} Notice that for any $t \in [0,T]$ we have the bound
	$$
	\E_x \left( M_t^2 \mathbbm{1}_{\{X_t \notin J_n\}}\right) \leq \frac{e^{2\lambda T}}{h^2(x)} \cdot \frac{1}{n^2} + \E_x \left( \left(\sup_{s \in [0,T]} M_s^2\right) \mathbbm{1}_{\{ \sup_{s \in [0,T]} h(X_s) > n\}}\right)  
	$$ so that it will suffice to show that 
	\begin{equation} \label{eq:lemaconv2}
	\lim_{n \rightarrow +\infty} \E_x \left( \left(\sup_{s \in [0,T]} M_s^2\right) \mathbbm{1}_{\{ \sup_{s \in [0,T]} h(X_s) > n\}}\right)=0.
	\end{equation} But since $\sup_{s \in [0,T]} M^2_s$ is $P_x$-integrable by Doob's inequality and (A1), and we also have that 
	$$
	\lim_{n \rightarrow +\infty} \mathbbm{1}_{\{ \sup_{s \in [0,T]} h(X_s) > n\}} = \mathbbm{1}_{\{ \sup_{s \in [0,T]} M_s^2 = +\infty\}},
	$$ where the right-hand side is now $P_x$-almost surely null by the integrability of $\sup_{s \in [0,T]} M^2_s$, using the dominated convergence theorem we can conclude \eqref{eq:lemaconv2}. 
\end{proof}

The second lemma concerns the asymptotic behavior of the function $p$ in \eqref{A2}.

\begin{lemma}\label{lema:p} The function $p$ from assumption (A2) satisfies:
	\begin{itemize}
		\item [i.] $p$ has subexponential growth, i.e. $\lim_{t \rightarrow +\infty} \frac{ \log p(t)}{t} = 0$.
		\item [ii.] If we define the function $q(t_1,t_2) := \frac{p(t_2)}{p(t_1+t_2)}$ then for any $C > 0$ we have that
		$$
		\limsup_{t_2 \rightarrow +\infty}\left[\sup_{t_1 \in [0,Ct_2]} q(t_1,t_2)\right]=:q_C < +\infty \hspace{1cm}\text{ and }\hspace{1cm}
		\lim_{t_2 \rightarrow +\infty} \left[\sup_{t_1 \in [0,C]} |q(t_1,t_2)-1|\right] = 0.
		$$
	\end{itemize}
\end{lemma}

\begin{proof} It is well-known that if $p$ is a regularly varying function at infinity then there exits $\alpha \in \R$ such that $p(t)= t^{\alpha} L(t)$ for some \textit{slowly} varying function $L$, i.e. a function $L :(0,+\infty) \rightarrow (0,+\infty)$ such that for any $a > 0$
	\begin{equation}
	\label{eq:limitingl}
	\lim_{t \rightarrow +\infty} \frac{L(at)}{L(t)} = 1.
\end{equation} Since it is straightforward to check that the function $t^\alpha$ satisfies (i) and (ii) above, it suffices to verify these claims for any slowly varying function $L$.
	
	To see this, we notice that by Karamata's representation theorem for any such $L$ there exists some $T > 0$ such that for every $t > T$ one has 
	$$
	L(t) = \exp\left\{ g_1(t) + \int_T^t \frac{g_2(s)}{s}ds\right\}
	$$ where $g_1,g_2 : (T,+\infty) \rightarrow \R$ are two bounded measurable functions which respectively satisfy
	$$
	\lim_{t \rightarrow +\infty} g_1(t) = b \in \R \hspace{2cm}\text{ and }\hspace{2cm}\lim_{t \rightarrow +\infty} g_2(t)=0.
	$$ From this representation (i) immediately follows. On the other hand, since the \mbox{convergence in \eqref{eq:limitingl}} is uniform whenever $a$ is restricted to compact sets of $(0,+\infty)$, (ii) now follows immediately from the writing
	$$
	\frac{L(t_2)}{L(t_1+t_2)} = \frac{L(t_2)}{L\left(t_2\left(1 + \frac{t_1}{t_2}\right)\right)} 
	$$ and \eqref{eq:limitingl}, using that $1+\frac{t_1}{t_2} \in [1,1+C]$. This concludes the proof.
\end{proof}

\subsection{Part I}\label{sec:part1} Assume first that $\Phi_x <+\infty$ and let us show that then one has
\begin{equation}\label{eq:phi1}
\lim_{t \rightarrow +\infty} \E_x(W^2_t(B,B'))=\left[\frac{\nu(B)}{\nu(B')}\right]^2 \Phi_x.
\end{equation} To this end, take $t > 0$ and notice that 
\begin{equation}
\label{eq:exp}
\E_x(W_t^2(B,B')) = \frac{\E_{x}(\xi^2_t(B))}{\E_x^2(\xi_t(B'))}.
\end{equation} Let us compute the expectations in the right-hand side of \eqref{eq:exp} by using the many-to-few lemmas. On the one hand, by the many-to-one lemma we have that
	\begin{equation}
	\label{eq:mto1}
	\E_x(\xi_t(B')) = \E_x \left( \sum_{u \in \overline{\chi}_t} \mathbbm{1}_{\{u_t \in B'\}} \right) = e^{r(m_1-1)t}P_x( X_t \in B').
	\end{equation} On the other hand, the many-to-two lemma yields
	\begin{align*}
	\E_x(\xi^2_t(B)) &= \E_x \left( \sum_{u,v \in \overline{\xi}_t} \mathbbm{1}_{\{u_t \in B\}}\mathbbm{1}_{\{v_t \in B\}}\right) \\
	& = e^{2r(m_1-1)t}\E_x \left(  \mathbbm{1}_{\{X^{(1)}_t \in B\}}\mathbbm{1}_{\{X^{(2)}_t \in B\}}e^{[\text{Var}(m)+(m_1-1)^2]r(E\wedge t)}\right).
	\end{align*} By separating in cases depending on whether $E > t$ or not, we obtain
	$$
	\E_x(\xi^2_t(B)) = (1)_{t} + (2)_{t},
	$$ where
	$$
	(1)_{t} := e^{r(m_2-1)t} \E_x (\mathbbm{1}_{\{X_t \in B\}} \mathbbm{1}_{\{E > t\}}) 
	$$ and 
	$$
	(2)_{t}: = e^{2r(m_1-1)t}\E_x\left(  \mathbbm{1}_{\{X^{(1)}_t \in B\}}\mathbbm{1}_{\{X^{(2)}_t \in B\}}e^{[\text{Var}(m)+(m_1-1)^2]rE}\mathbbm{1}_{\{E \leq t\}}\right).
	$$ Now, using the independence of $E$ from the motion of the $2$-spine, the Markov property yields
	$$
	(1)_{t}= e^{r(m_1-1)t}P_x(X_t \in B)
	$$ and
	$$
	(2)_{t}= (m_2-m_1)re^{2r(m_1-1)t} \int_0^t P_x\left(X^{(1),s}_t \in B, X^{(2),s}_t \in B\right)e^{-r(m_1-1)s}ds,
	$$ where $X^{(1),s}$ and $X^{(2),s}$ are two coupled copies of the Markov process $X$ which coincide \mbox{until time $s$} and then evolve independently after $s$. If we condition on the position of these coupled processes at time $s$, then we obtain
	\begin{equation} \label{eq:2th}
	(2)_{t} = (m_2-m_1)re^{2r(m_1-1)t} \int_0^t \E_x( P^2_{X_s}(X_{t-s} \in B) )e^{-r(m_1-1)s}ds.
	\end{equation} Now, from \eqref{eq:mto1} and \eqref{A2} we conclude that
	$$
	[1]_{t}:= \frac{(1)_{t}}{\E^2_x(\xi_t(B'))} = \frac{P_x(X_t \in B)}{P_x(X_t \in B')} \cdot \frac{1}{\E_x(\xi_t(B'))} = \frac{\nu(B)+s_B(x,t)}{[\nu(B')+s_{B'}(x,t)]^2} \cdot \frac{1}{h(x)p(t)}e^{-(r(m_1-1)-\lambda)t}
	$$ which, by (i) in Lemma \ref{lema:p}, shows that if $t$ is taken sufficiently large then 
	\begin{equation}
	\label{eq:bound1}
	|[1]_t| \leq 2 \frac{\nu(B)}{\nu(B')} \frac{e^{-\frac{1}{2}(r(m_1-1)-\lambda)t}}{h(x)}.
	\end{equation} Similarly, one has that
	$$
	[2]_{t}:=\frac{(2)_{t}}{\E^2_x(\xi_t(B))}= \int_0^t \Psi_{x,t}(s) ds,
	$$
	where
	$$
	\Psi_{x,t}(s):=(m_2-m_1)r \frac{\E_x( P^2_{X_s}(X_{t-s} \in B))}{P^2_x( X_t \in B')}e^{-r(m_1-1)s}.
	$$ To treat the term $[2]_{t}$ we split the integral into three separate parts, i.e. for $\alpha \in (0,1)$ and $T > 0$ to be specified later we write
	$$
	[2]_{t} = [a]_{t} + [b]_{t} + [c]_t
	$$ where
	$$
	[a]_{t} := \int_{\alpha t}^t \Psi_{x,t}(s)ds \hspace{1cm} 
	[b]_{t} := \int_{T}^{\alpha t} \Psi_{x,t}(s)ds \hspace{1cm}[c]_t:=\int_0^T \Psi_{x,t}(s)ds  . 
	$$ The first term $[a]_{t}$ deals with the case in which $s \rightarrow t$ and the asymptotics in \eqref{A2} for $P_{y}(X_{t-s} \in B)$ may not hold. In this case, if $\alpha$ is taken close enough to $1$ then $[a]_{t}$ tends to zero as $t \rightarrow +\infty$. 
	Indeed, notice that 
	$$
	\E_x( P_{X_s}^2(X_{t-s} \in B) ) \leq \E_x(P_{X_s}( X_{t-s} \in B)) = P_x( X_t \in B)
    $$ by the Markov property, so that
    $$
    [a]_{t} \leq (m_2-m_1)r \frac{P_x(X_t \in B)}{P_x^2(X_t \in B')} \int_{\alpha t}^\infty  e^{-r(m_1-1)s} ds = \frac{m_2-m_1}{m_1-1}e^{(1-\alpha)r(m_1-1)t} [1]_{t}.
    $$
	By recalling \eqref{eq:bound1}, if $\alpha$ is chosen sufficiently close to $1$ and $t$ taken sufficiently large then 
	\begin{equation}\label{eq:bounda}
	|[a]_{t}| \leq e^{-\frac{1}{4}(r(m_1-1)-\lambda)t}.
	\end{equation} Similarly to $[a]_{t}$, the term $[b]_t$ can also be made arbitrarily small as $t \rightarrow +\infty$ if $T$ is large enough. Indeed, by \eqref{A2} we have that
	\begin{align}
	\Psi_{x,t,h}(s) &= (m_2-m_1)r  \left[\frac{q(s,t-s)}{\nu(B')+s_{B'}(x,t)}\right]^2 \E_x(M_s^2(\nu(B)+s_B(X_s,t-s))^2)e^{-r(m_1-1)s} \label{eq:decomp}
	\\
	& \leq (m_2-m_1)8rq^2_{\frac{\alpha}{1-\alpha}} \left(\frac{\nu(B)+\overline{s}_B}{\nu(B')}\right)^2 \E_x(M^2_s) e^{-r(m_1-1)s} \nonumber
	\end{align}
	if $s \leq \alpha t$ and $t$ is large enough so as to have
	\begin{itemize}
		\item [$\bullet$] $\sup_{y \in J,\,u\geq(1-\alpha)t} s_B(y,u) \leq \overline{s}_B$,
		\item [$\bullet$] $q(s,t-s) \leq 2q_{\frac{\alpha}{1-\alpha}}$ (which can be done by (ii) in Lemma \ref{lema:p} since $\frac{s}{t-s}\leq \frac{\alpha}{1-\alpha}$),
		\item [$\bullet$] $s_{B'}(x,t) \geq -\frac{\nu(B')}{2}$,
	\end{itemize} so that
	\begin{equation} \label{eq:bcot}
	|[b]_{t}| \leq 8q^2_{\frac{\alpha}{1-\alpha}} \left(\frac{\nu(B)+\overline{s}_B}{\nu(B')}\right)^2 \cdot (m_2-m_1) \int_T^\infty \E_x(M^2_s)re^{-r(m_1-1)s}ds.
	\end{equation} Since $\Phi_x < +\infty$, the right-hand side of \eqref{eq:bcot} can be made arbitrarily small if $T$ is chosen sufficiently large depending on $\alpha$. 
	
	Finally, let us treat the last term $[c]_{t}$. By (A2) and (ii) in Lemma \ref{lema:p}, for $s \leq T$ we may write
	$$
	\Psi_{x,t,h}(s) = \frac{(m_2-m_1)r}{\nu^2(B')} (1 + o_t(1)) \E_x(M_s^2(\nu(B)+s_B(X_s,t-s))^2)e^{-r(m_1-1)s}
	$$ where $o_t(1)$ (which depends on $x,t,s$ and $B'$) tends to zero uniformly in $s \leq T$ as $t \rightarrow +\infty$. Thus, we may decompose
	$$
	[c]_{t}=[c_1]_{t} + [c_1^*]_{t}
	$$ with 
	$$
	[c_1^*]_{t}= \frac{m_2-m_1}{\nu^2(B')} \int_0^T \E_x(M_s^2(\nu(B)+s_B(X_s,t-s))^2))re^{-r(m_1-1)s}ds
	$$ and 
	\begin{equation} \label{eq:ccot1}
	|[c_1]_{t}| \leq \frac{(\nu(B)+\overline{s}_B)^2}{\nu^2(B')} \Phi_x \left[\sup_{s \leq T}o_t(1)\right], 
	\end{equation} where the right-hand side of \eqref{eq:ccot1} tends to zero as $t \rightarrow +\infty$ since $\Phi_x < +\infty$. Finally, given $n \in \N$ we decompose $[c_1^*]_{t}$ by splitting the expectation inside into two depending on whether $X_s \in J_n$ or not. More precisely, we write 
	$$
	[c_1^*]_{t}=[c_2]_{t}+[c_2^*]_{t}
	$$ where 
	$$
	[c_2^*]_{t} = \frac{(m_2-m_1)}{\nu^2(B')}\int_0^T \E_x(M_s^2(\nu(B)+s_B(X_s,t-s))^2 \mathbbm{1}_{\{X_s \in J_n\}})re^{-r(m_1-1)s}ds
	$$ and 
	\begin{equation}
	\label{eq:ccot2}
	|[c_{2}]_{t}| \leq \frac{m_2-m_1}{m_1-1} \frac{(\nu(B)+\overline{s}_B)^2}{\nu^2(B')} \sup_{s \in [0,T]} \E_x(M_s^2\mathbbm{1}_{\{X_s \notin J_n\}}).
	\end{equation} Notice that the right-hand side of \eqref{eq:ccot2} is independent of $t$ and tends to zero as $n$ tends to infinity for any fixed $T > 0$ by Lemma \ref{lema:A3}. On the other hand, observe that by (A2) and Lemma \ref{lema:A3}
	$$
	\E_x(M_s^2(\nu(B)+s_B(X_s,t-s))^2) \mathbbm{1}_{\{X_s \in J_n\}}) = \nu^2(B) \E_x(M_s^2)(1+\overline{o}_t(1))
	$$ where the term $\overline{o}_t$ (which depends on $x,n,t,s$ and $B$) tends to zero uniformly in $s \leq T$ as $t \rightarrow +\infty$ since $\sup_{s \in [0,T]} \E_x(M_t^2) \leq 4\E_x(M_T^2)$ by Doob's inequality. By repeating the same argument that lead us to \eqref{eq:ccot1}, we conclude that 
	$$
	[c_{2}^*]_{t,h}=[c_{3}]_{t} + [c_{4}]
	$$ where 
	$$
	[c_4] = \left[\frac{\nu(B)}{\nu(B')}\right]^2 \cdot (m_2-m_1)\int_0^T \E_x(M_s^2)re^{-r(m_1-1)s}ds
	$$ and $|[c_3]_{t}|$ tends to zero as $t \rightarrow +\infty$. Thus, we find that if we write $\Gamma:=\{1,a,b,c_1,c_2,c_3\}$ then 
	\begin{equation}
	\label{eq:cotfinal}
	\left| W^{(x)}_{t}(B,B') - \left[\frac{\nu(B)}{\nu(B')}\right]^2\Phi_x\right| \leq \sum_{i \in \Gamma} |[i]_t| + \left[\frac{\nu(B)}{\nu(B')}\right]^2 \cdot (m_2-m_1)\int_T^\infty \E_x(M_s^2)re^{-r(m_1-1)s}ds.
	\end{equation} By taking $\alpha$ adequately close to $1$, $T$ large enough (depending on $\alpha$) and then $n$ sufficiently large (depending on $T$), the right-hand side of \eqref{eq:cotfinal} can be made arbitrarily small for all $t$ large enough and so \eqref{eq:phi1} follows. 
	
	Now, let us assume that $\Phi_x = +\infty$ and show that $\lim_{t \rightarrow +\infty} \E_x(W^2_t(B,B'))=+\infty$ in this case, proving \eqref{eq:phi1}. To see this, we notice that for any fixed $T > 0$ we have
	$$
	\E_x\left(W^2_t(B,B')\right) \geq [c]_{t,0} = \int_0^T \Psi_{x,t,0}(s)ds.
	$$ If $n$ is chosen large enough so that $\sup_{s \in [0,T]} \E_x( M_s^2 \mathbbm{1}_{\{X_s \notin J_n\}}) < 1$, then for all $t$ large enough to  guarantee that
	\begin{enumerate}
		\item [$\bullet$] $\inf_{s \in [0,T]} \left( \frac{q(s,t-s)}{\nu(B')+s_{B'}(x,t)}\right)^2 \geq \frac{1}{2\nu^2(B')}$
		\item [$\bullet$] $\inf_{s \in [0,T], y \in J_n} (\nu(B)+s_B(y,t-s))^2 \geq \frac{\nu^2(B)}{2},$
	\end{enumerate} by \eqref{eq:decomp} we obtain 
	\begin{align*}
	[c]_{t,0} &\geq \left[\frac{\nu(B)}{\nu(B')}\right]^2\frac{(m_2-m_1)}{4} \int_0^T \E_x(M_s^2 \mathbbm{1}_{\{X_s \in J_n\}})re^{-r(m_1-1)s}ds\\
	\\
	&\geq \left[\frac{\nu(B)}{\nu(B')}\right]^2\frac{(m_2-m_1)}{4} \int_0^T \E_x(M_s^2) re^{-r(m_1-1)s}ds - \left[\frac{\nu(B)}{\nu(B')}\right]^2\frac{(m_2-m_1)}{4(m_1-1)}.
	\end{align*}The right-hand side of this last inequality can be made arbitrarily large by taking $T$ big enough, due to the fact that $\Phi_x=+\infty$. In particular, this implies that 
	$$
	\lim_{t \rightarrow +\infty}\E_x\left(W^2_t(B,B')\right)=+\infty
	$$ and thus concludes the proof of Part I.
	
\subsection{Part II}

We now check that, whenever $\Phi_x < +\infty$, one has 
		$$
		W^{(x)}_t(B,B') \overset{L^2}{\longrightarrow} \frac{\nu(B)}{\nu(B')} \cdot D^{(x)}_\infty.
		$$ for every $B,B' \in \mathcal{C}_X$ with $\nu(B')> 0$.  
Notice that by Theorem \ref{theo:main2} it suffices to show that  
		\begin{equation} \label{eq:kseq}
		\lim_{t \rightarrow +\infty} \left\| W^{(x)}_t(B,B') - \frac{\nu(B)}{\nu(B')} \cdot D^{(x)}_t \right\|_{L^2} = 0.
		\end{equation} Now, observe that
		$$
		\left\| W^{(x)}_t(B,B') - \frac{\nu(B)}{\nu(B')} \cdot D^{(x)}_t \right\|_{L^2}^2 = \E_x(W^2_t(B,B')) - 2\frac{\nu(B)}{\nu(B')}\E_x(W_t(B,B')D_t) + \left[\frac{\nu(B)}{\nu(B')}\right]^2\E_x(D^2_t)
		$$ so that, by \eqref{eq:phi1} and Theorem \ref{theo:main2}, \eqref{eq:kseq} will follow if we show that
		$$
		\lim_{t \rightarrow +\infty} \E_x(W_t(B,B')D_t) = \frac{\nu(B)}{\nu(B')} \Phi_x.
		$$ But this can be done by proceeding exactly as in Part I. We omit the details.
		
		Finally, that 
		$$
		\nu^{(x)}_t(B,B') \overset{P}{\longrightarrow} \frac{\nu(B)}{\nu(B')}
		$$ conditionally on the event $\{D^{(x)}_\infty > 0\}$ follows from the fact that 
		\begin{equation}\label{eq:convpexp}
		\frac{\E_x(\xi_t(B'))}{\xi^{(x)}_t(B')}=\frac{1}{W^{(x)}_t(B',B')}\overset{P}{\longrightarrow}\frac{1}{D^{(x)}_\infty} \hspace{1cm}\text{ and }\hspace{1cm}\frac{\xi^{(x)}_t(B)}{\E_x(\xi_t(B'))}=W^{(x)}_t(B,B') \overset{P}{\longrightarrow} \frac{\nu(B)}{\nu(B')} \cdot D^{(x)}_\infty
		\end{equation} conditionally on the event $\{D^{(x)}_\infty > 0\}$, which in turn follows from \eqref{eq:conv1}. This concludes Part II and thus the proof of Theorem \ref{theo:main}.
		

\section{Proof of Theorem \ref{theo:main3}}\label{sec:theo3}

We also divide the proof of Theorem \ref{theo:main3}, now into four parts. First, we show that $\eta$ and $\sigma$ are indeed fixed points of $G$ and, using (B3), that $\eta$ and $\mathbf{1}$ cannot intersect other fixed points of $G$. Next, we show that (B1-B2) imply that our branching dynamics can be  dominated from below by a supercritical Galton-Watson process. Using this domination, we then show that the notion of strong supercriticality can be reformulated in terms of certain fixed points of the operator $G$. Finally, we use this alternative formulation to show both implications of Theorem \ref{theo:main3}.

\subsection{Part I}

We begin by checking first that both $\eta$ and $\sigma$ are fixed points of $G$.

\begin{proposition} \label{prop:fp}
	The functions $\eta$ and $\sigma$ are fixed points of $G$.	
\end{proposition}

\begin{proof}
	Observe that for any $t > 0$ and $x \in J$ we have the relation
	$$
	|\xi_{1+t}^{(x)}| = \sum_{u \in \overline{\xi}^{(x)}_1} |\xi_t^{(u_1)}|, 
	$$ which implies that, for any $t > 0$, $|\xi^{(x)}_{1+t}|$ equals zero if and only if $|\xi_t^{(u_1)}|$ is zero for every $u \in \overline{\xi}_1^{(x)}$. Thus, if we take $t \rightarrow +\infty$ then the former yields 
	\begin{equation}
	\label{eq:eta1}
	\mathbbm{1}_{\left(\Theta^{(x)}\right)^c} = \prod_{u \in \overline{\xi}^{(x)}_1}\mathbbm{1}_{\left(\Theta^{(u_1)}\right)^c}.
	\end{equation} By taking expectations $\E_x$ on the equality in \eqref{eq:eta1}, we obtain that $\eta(x)=G(\eta)(x)$. Furthermore, since this holds for any $x \in J$, we conclude that $\eta$ is a fixed point of $G$.
	
	Now, to see that $\sigma$ is a fixed point of $G$, we observe the analogous relation
	$$
	D_{1+t}^{(x)} = \frac{1}{h(x)}\sum_{u \in \overline{\xi}^{(x)}_1} h(u_1) e^{-(r(m_1-1)-\lambda)} D_t^{(u_1)}
	$$ which, upon taking the limit $t \rightarrow +\infty$, becomes
	$$
	D_{\infty}^{(x)} = \frac{1}{h(x)}\sum_{u \in \overline{\xi}^{(x)}_1} h(u_1) e^{-(r(m_1-1)-\lambda)} D_\infty^{(u_1)}.
	$$ Since $h(y)=0$ if and only if $y \in \partial_* \overline{J}$, that $\sigma$ is a fixed point of $G$ now follows as before.  
\end{proof}

Next, we use irreducibility to see that $\eta$ and $\mathbf{1}$ cannot intersect other fixed points of $G$.

\begin{proposition}\label{prop:fpG}
	Assume that (B3) holds. Then, if $g$ is a fixed point of $G$ we have that:
	\begin{enumerate}
		\item [i.] $\eta(x) \leq g(x) \leq 1$ for all $x \in J$.
		\item [ii.] $g(x) = \eta(x)$ for some $x \in J \Longrightarrow g \equiv \eta$.
		\item [iii.] $g(x) = 1$ for some $x \in J \Longrightarrow g \equiv \mathbf{1}$. 
	\end{enumerate}
\end{proposition} 

\begin{proof} We show first that if $g$ is a fixed point of $G$ then $\eta \leq g \leq \mathbf{1}$. Indeed, the $g \leq \mathbf{1}$ inequality is immediate whereas the $\eta \leq g$ inequality follows from the fact that $G$ is an increasing operator, i.e. $G(f_1) \leq G(f_2)$ if $f_1 \leq f_2$, together with the fact that $\eta = \lim_{n \rightarrow +\infty} G^{(n)}(\mathbf{0})$, where $G^{(n)}$ denotes the $n$-th composition of $G$ with itself and $\mathbf{0}$ is the function constantly equal to $0$.

Now, let us prove (ii). First, we observe that it is easy to check by induction that for any $n \in \N$ 
$$
G^{(n)}(g)(x)= \E_x \left( \prod_{u \in \overline{\xi}_n} g(u_n) \right).
$$ In particular, if $x \in J$ satisfies $P_x( X_n \in \{ y \in J : \eta(y) < g(y) \}) > 0 ) > 0$ for some $n \in \N$ then, by considering only evolutions of $\xi^{(x)}$ in which there is no branching until time $n$, it is clear that 
$$
P_x\left( \prod_{u \in \overline{\xi}_n} \eta(u_n) < \prod_{u \in \overline{\xi}_n} g(u_n)\right) > 0
$$ so that 
$$
\eta(x)=G^{(n)}(\eta)(x)=\E_x \left( \prod_{u \in \overline{\xi}_n} \eta(u_n)\right) < 
\E_x \left( \prod_{u \in \overline{\xi}_n} g(u_n)\right)=G^{(n)}(h)(x)=g(x).
$$ Therefore, if $\eta(x)=g(x)$ then we must have $P_x( X_n \in \{ y \in J : \eta(y) < g(y) \}) = 0$ for every $n \in \N$. By irreducibility (assumption (B3)) we then obtain that $P_{x'}( X_{1} \in \{ y \in J : \eta(y) < g(y) \}) = 0$ and, as a consequence, that $P_{x'}( \xi_{1}( \{ y \in J: \eta(y) < g(y) \}) > 0 ) = 0$ holds for every $x' \in J$ since the particle positions $(u_1)_{u \in \overline{\xi}^{(x')}_1}$ are all distributed as $X^{(x')}_1$. Since $\eta$ and $h$ are fixed points of $G$, this implies that $g(x') \leq \eta(x')$ for every $x' \in J$ which, together with (i) shown above, allows us to conclude that $\eta \equiv g$. The proof of (iii) is analogous.
\end{proof}

\subsection{Part II} The following step is to show that, under (B1-B2), one has a suitable lower bound on the growth of our dynamics. To this end, for each $n \in \N$ we define the set 
\begin{equation}
\label{eq:defjotatilde}
\tilde{J}_n:=\{ x \in J : \Phi_x \leq n\}
\end{equation} and then write $\hat{J}_n:=J_n \cap \tilde{J}_n$. Notice that the sequence $(\hat{J}_n)_{n \in \N}$ is increasing and, furthermore, that $\cup_{n \in \N} \hat{J}_n = J$ by (B1). Now, the precise meaning of lower bound on the growth of our dynamics is formulated in the following definition. 

\begin{definition} \label{def:bstar} We say that the \textit{lower bound condition} holds, and denote it in the \mbox{sequel by (LB),} if for any $n \in \N$ and $B \in \B_J$ with $\nu(B) > 0$ there exists a time $T_{n,B} $ and a random variable $L_{n,B}$ satisfying $\E(L_{n,B}) > 1$ such that for all $x \in \hat{J}_n$ and every $t > T_{n,B}$ one has
	$$
	L_{n,B} \preceq \xi^{(x)}_t(B)
	$$ where $\preceq$ denotes stochastic domination, i.e. for any bounded measurable and increasing $f: \R \to \R$ one has that
	$$
	\E(f(L_{n,B})) \leq \inf_{x \in \hat{J}_n} \E_x( f(\xi_t(B))).
	$$ 
\end{definition}

\begin{remark}\label{rem:crec} Note that, by (B2) and Lemma \ref{lema:mt1} below, for any \mbox{$B \in \B_J$ with $\nu(B) > 0$ and $x \in J_n$} we have that 
	\begin{align}
	\E_x(\xi_t(B)) &\geq \E_x(\xi_t(B^*)) = h(x) p(t) e^{(r(m_1-1)-\lambda)t}(\nu(B^*)+s_{B^*}(x,t)) \nonumber\\
	& \geq \frac{1}{n}p(t) e^{(r(m_1-1)-\lambda)t}\left(\nu(B^*)+ \inf_{y \in J_n} s_{B^*}(y,t)\right) \label{eq:assump2B1}
	\end{align} so that, by Lemma \ref{lema:p} and (A2), for all $t$ large enough depending on $B$ and $n$ we have that
	\begin{eqnarray}
	\label{eq:assump2B}
	\inf_{x \in \hat{J}_n} \E_x(\xi_t(B)) > 1.
	\end{eqnarray} Thus, condition (LB) is simply a stronger form of \eqref{eq:assump2B}, one in which we ask the entire distributions of the random variables $(\xi_t^{(x)}(B))_{x \in \hat{J}_n}$ to be uniformly supercritical rather than just their means. 
\end{remark}

The lower bound (LB) will be the main tool in the proof of Lemma \ref{lemma:cru} in Part III below, which is crucial for proving Theorem \ref{theo:main3}. Our next result states that (LB) holds under (B2).  

\begin{proposition}\label{prop:b1} Assumption (B2) implies condition (LB). 
\end{proposition}

\begin{proof}
	Let us fix $n \in \N$ and notice that, by the Cauchy-Schwarz inequality, we have for any $x \in J_n$, $B \in \B_{J}$ and $K,T \in \N$ that
	\begin{equation} \label{eq:holder}
	\E_{x}^2(\xi_T(B)\mathbbm{1}_{\xi_T(B)\geq K}) \leq \E_x(\xi^2_T(B))P_x(\xi_T(B)\geq K).
	\end{equation}On the other hand, if $\nu(B)>0$ then it follows from (B2), \eqref{eq:assump2B1} and Assumptions \ref{assumpG} that 
	$$
	\lim_{T \rightarrow +\infty} \left[ \inf_{y \in J_n} \E_y(\xi_T(B))\right] = +\infty.
	$$ Therefore, by \eqref{eq:holder} we conclude that if $T$ is sufficiently large (depending only on $K$, $n$ and $B$) then for all $x \in J_n$ we have
	$$
	P_x(\xi_T(B)\geq K) \geq \frac{[\E_x(\xi_T(B))-K]^2}{\E_x(\xi^2_T(B))} \geq \frac{1}{2} \cdot \frac{\E_x^2(\xi_T(B))}{\E_x(\xi_T^2(B))}.
	$$ Now, a careful inspection of the proof of Theorem \ref{theo:main} shows that there exists a constant $C_n > 0$ and a time $T_n > 0$ such that 
	$$
	\frac{\E_x(\xi^2_T(B))}{\E_x^2(\xi_T(B))} \leq C_n \Phi_x + 1
	$$ for all $x \in J_n$ and $T > T_n$. We stress that $C_n$ and $T_n$ do \textit{not} depend on $x \in J_n$, only on $n$ and $B$. Therefore, since $\sup_{y \in \tilde{J}_n} \Phi_y <+\infty$, we may take $K \in \N$ sufficiently large and $T \in \N$ accordingly so that 
	$$
	\inf_{x \in \hat{J}_n} P_x(\xi_T(B)\geq K ) \geq \frac{1}{K-1}.
	$$ It follows that $L_{n,B} \preceq \xi^{(x)}_T(B)$ for any such $T$ and all $x \in \hat{J}_n$, where $L_{n,B}$ has distribution given by 
	$$
	P(L_{n,B}=K) = \frac{1}{K-1} = 1 - P(L_{n,B}=0).
	$$ Since in this case $\E(L_{n,B})=\frac{K}{K-1}>1$, this concludes the proof.
\end{proof}

\subsection{Part III}

We continue by using Proposition \ref{prop:b1} to show that strong supercriticality can be reformulated \mbox{in terms of} certain fixed points of $G$. More precisely, we have the following result.

\begin{proposition} \label{prop:ss1} If Assumptions \ref{assumpG2} are satisfied then $\xi^{(x)}$ is strongly supercritical if and only if the following two conditions hold:
	\begin{enumerate}
		\item [i.] $\xi^{(x)}$ is supercritical, i.e. $P_x(\Theta) > 0$.
		\item [ii'.] There exists $n \in \N$ such that 
	$$
	P_x(\Theta) = P_x \left( \limsup_{k \rightarrow +\infty} \xi_k(\tilde{J}_n) > 0\right)
	$$ where $\tilde{J}_n$ is given by \eqref{eq:defjotatilde}.
	\end{enumerate}
\end{proposition}

\begin{remark}\label{rem:fpG}Observe that for $B \in \B$ the function $g_{B}$ defined as
	$$
	g_{B}(x):= P_x\left( \limsup_{n \rightarrow +\infty} \xi_{n}(B) = 0\right)
	$$ is a fixed point of $G$. Indeed, the proof of this statement is analogous to that of Proposition \ref{prop:fp}. Thus, Proposition \ref{prop:ss1} states that $\xi^{(x_0)}$ is strongly supercritical if and only if for some $n \in \N$
	$$
	g_{\tilde{J}_n}(x_0)=\eta(x_0)<1.
	$$ But then by Proposition \ref{prop:fpG} we conclude that the same statement must hold for all $x \in J$, so that $\xi^{(x)}$ is strongly supercritical for some $x \in J$ if and only if it is strongly supercritical for all $x \in J$.
\end{remark}

We now prove Proposition \ref{prop:ss1}. Let us notice that it suffices to show that (ii) in Definition \ref{def:ss} is equivalent to (ii') in the statement above. To see the (ii')$\Longrightarrow$(ii) implication, notice the inclusions
\begin{equation} \label{eq:traninc}
\left(\Theta^{(x)}\right)^c \subseteq \Gamma^{(x)} = \bigcap_{n \in \N} \left\{ \lim_{t \rightarrow +\infty} \xi^{(x)}_{t}(\tilde{J}_n) = 0 \right\} \subseteq \bigcap_{n \in \N} \left\{ \lim_{k \rightarrow +\infty} \xi^{(x)}_{k}(\tilde{J}_n) = 0 \right\}.
\end{equation} Now, for each $n \in \N$ let us write $A_{n}^{(x)}:=\left\{ \lim_{k \rightarrow +\infty} \xi^{(x)}_{k}(\tilde{J}_n) = 0 \right\}$. Notice that, since the sequence $(\tilde{J}_n)_{n \in \N}$ is increasing, we have that $(A^{(x)}_{n})_{n \in \N}$ is decreasing and therefore that
\begin{equation} \label{eq:traninc2}
P_x\left(\bigcap_{n \in \N} \left\{ \lim_{k \rightarrow +\infty} \xi_{k}(\tilde{J}_n) = 0 \right\}\right) = \lim_{n \rightarrow +\infty} P_x(A_{n}).
\end{equation} Therefore, if (ii') holds then it follows that $P_x(\Theta^c) = P_x( A_{n} )$ for some $n$ and, by \eqref{eq:traninc} and \eqref{eq:traninc2}, we conclude that (ii) holds. On the other hand, if (ii) holds then by \eqref{eq:traninc} we have 
\begin{equation} \label{eq:tran3}
P_x(\Theta^c)= \lim_{n \rightarrow +\infty} P_x \left( \lim_{t \rightarrow +\infty} \xi_{t}(\tilde{J}_n)= 0\right).
\end{equation} Thus, if we show that for all $n \in \N$ sufficiently large so that $x \in \hat{J}_n$ and $\nu(\tilde{J}_n) > 0$ we have  
\begin{equation} \label{eq:tranimpl}
\lim_{k \rightarrow +\infty} \xi^{(x)}_{k}(\tilde{J}_n) = 0 \Longrightarrow \lim_{t \rightarrow +\infty} \xi^{(x)}_{t}(\tilde{J}_{n+1}) = 0
\end{equation} then, from \eqref{eq:tran3} and the inclusion $\left(\Theta^{(x)}\right)^c \subseteq \{ \lim_{k \rightarrow +\infty} \xi^{(x)}_{k}(\tilde{J}_n) = 0\}$, by iterating \eqref{eq:tranimpl} we conclude that for $n$ sufficiently large 
$$
P_x(\Theta^c)\leq P_x\left( \lim_{k \rightarrow +\infty} \xi_{k}(\tilde{J}_n) = 0\right) \leq \lim_{m \rightarrow +\infty} P_x \left( \lim_{t \rightarrow +\infty} \xi_{t}(\tilde{J}_m)= 0\right) = P_x(\Theta^c),
$$ which immediately gives (ii'). Now, \eqref{eq:tranimpl} follows at once from the next lemma.

\begin{lemma} \label{lemma:cru} For any $m \in \N$ and $B \in \B_J$ such that $\hat{J}_m \neq \emptyset$ and $\nu(B)>0$ there exists $T \in \N$ satisfying that for any $x \in \hat{J}_m$ one has
	$$
	\limsup_{t \rightarrow +\infty} \xi^{(x)}_t(\tilde{J}_m) > 0 \Longrightarrow \lim_{k \rightarrow +\infty} \xi^{(x)}_{kT}(B) = +\infty.
	$$
\end{lemma}

\begin{proof} The idea is to couple the sequence $(\xi^{(x)}_{kT}(B))_{k \in \N}$ together with an i.i.d. sequence $(Z^{(n)})_{n \in \N}$ of supercritical single-type Galton-Watson branching processes such that if at least one $Z^{(n)}$ survives on the event $\{\limsup_{t \rightarrow +\infty} \xi^{(x)}_t(\tilde{J}_m) > 0\}$ then $\xi^{(x)}_{kT}(B)$ tends to infinity as $k \rightarrow +\infty$. 
	
We proceed as follows. First we notice that, since condition (LB) holds by Proposition \ref{prop:b1}, there exists a random variable $L_{m,B}$ with $\E(L_{m,B})>1$ and a time $T \in \N$ such that for all $y \in \hat{J}_m$ and $t \geq T$ we have
\begin{equation} \label{eq:domi}
L_{m,B} \preceq \xi^{(y)}_t(B).
\end{equation} Next, given a fixed $x \in \hat{J}_m$, we define the process $V^{(1)}:=(V^{(1)}_j)_{j \in \N}$ by the formula
$$
V^{(1)}_j:= \xi^{(x)}_{jT}(B)
$$ and observe that for each $j \in \N$ we have
$$
V^{(1)}_{j+1} \geq \sum_{u \in \overline{\xi}^{(x)}_{jT}(B)} \xi^{(u)}_T(B),
$$ where $\overline{\xi}^{(x)}_{jT}(B)$ denotes the subcollection of particles of $\overline{\xi}^{(x)}_{jT}$ which are located inside the subset $B$ and, for $u \in \overline{\xi}^{(x)}_{t}$, $\xi^{(u)}$ is the sub-dynamics of $\xi^{(x)}$ associated to the particle $u$ starting at time $t$. 
Since for each $u \in \overline{\xi}^{(x)}_{jT}(B)$ we have that $L_{B,m} \preceq \xi^{(u)}_T(B)$ by \eqref{eq:domi}, it follows that (by enlarging the current probability space if necessary) one can couple $V^{(1)}$ with a Galton-Watson process $Z^{(1)}:=(Z^{(1)}_j)_{j \in \N}$ with offspring distribution given by $L_{m,B}$, in such a way that $V^{(1)}_j \geq Z^{(1)}_j$ holds for all $j \in \N$. Therefore, if $Z^{(1)}$ survives then $Z^{(1)}_j$ must tend to infinity as $j \rightarrow +\infty$ and, consequently, so must $\xi^{(x)}_{jT}(B)$. In this case, we decouple $\xi^{(x)}$ from the remaining $Z^{(n)}$ (for $n \geq 2$) by taking these to be independent from $\xi^{(x)}$. If $Z^{(1)}$ dies out, however, we proceed as follows:
\begin{enumerate}
	\item [i.] Define $\tau^{(1)}:=\inf \{ j \in \N : Z^{(1)}_j = 0\}$. Notice that $Z^{(1)}$ dies out if and only if $\tau^{(1)}<+\infty$.
	\item [ii.] If $\xi^{(x)}_{sT}(\tilde{J}_m)=0$ for all $s \geq \tau^{(1)}$, then decouple the $(Z^{(n)})_{n \geq 2}$ from $\xi^{(x)}$ as before.
	\item [iii.] If $\xi^{(x)}_{sT}(\tilde{J}_m)>0$ for some (random) $s \geq \tau^{(1)}$, then choose some $y \in \overline{\xi}^{(x)}_{sT}(\tilde{J}_m)$ at random and define the process $V^{(2)}=(V^{(2)}_j)_{j \in \N}$ according to the formula
	$$
	V^{(2)}_{j}:= \xi^{(y)}_{(\lceil s\rceil+j)T}(B),
	$$ where $\lceil s \rceil$ here denotes the smallest integer greater than or equal to $s$. Let us observe that, by construction, there exists a (random) $k \in \N$ such that $V^{(2)}_j \leq \xi^{(x)}_{(k+j)T}(B)$ for all $j \in \N$. By a similar argument than the one carried out for $V^{(1)}$, it is possible to couple $V^{(2)}$ with a Galton-Watson process $Z^{(2)}$ which is independent of $Z^{(1)}$ but has the same distribution, in such a way that $V^{(2)}_j \geq Z^{(2)}_j$ for all $j \in \N$. If $Z^{(2)}$ survives then, by the considerations above, $\xi^{(x)}_{jT}$ must tend to infinity as $j \rightarrow +\infty$. If not, then one can repeat this procedure to obtain a branching process $Z^{(3)}$ and so on.	
\end{enumerate} Since every $Z^{(n)}$ has the same \textit{positive} probability of survival, it follows that at least one of them will survive on the event $\{\limsup_{t \rightarrow +\infty} \xi^{(x)}_t(\tilde{J}_m) > 0\}$, and so the result now follows.
\end{proof}

\subsection{Part IV}

We now conclude by showing all implications in the statement of Theorem \ref{theo:main3}. 

First, let us observe that the condition $\Phi_x < +\infty$ implies that $\sigma(x)<1$. Indeed, if $\Phi_x < +\infty$ then by Theorem \ref{theo:main2} we have $\E_x(D_\infty)=1$ so that  $\sigma(x)=P_x(D_\infty=0) < 1$ necessarily holds. Thus, if (B1) holds then $\sigma \neq \mathbf{1}$ and therefore (i) must imply (ii). 

\mbox{That (ii) $\Longrightarrow$ (iii)} is obvious, so we move on to (iii) $\Longrightarrow$ (iv). Take $x \in J$ such that $\eta(x)=\sigma(x)$. Note that, by the argument given above, $\sigma(x)<1$ so that if $\eta(x)=\sigma(x)$ then $\xi^{(x)}$ is supercritical. It remains to verify (ii') of Proposition \ref{prop:ss1}. But (B2) together with Theorem \ref{theo:main} and \eqref{eq:assump2B1} imply that for any $B \in \B_J$ with $\nu(B)> 0$
$$
\limsup_{n \rightarrow +\infty} P_{x} ( \xi_n(B) = 0 | \Lambda )  \leq \limsup_{n \rightarrow +\infty} P_{x} ( \xi_n(B^*) = 0 | \Lambda ) = 0, 
$$ from which a straightforward calculation yields that  
$$
P_x \left( \left\{ \limsup_{n \rightarrow +\infty} \xi_n(B) > 0 \right\} \cap \Lambda \right) = P_x(\Lambda).
$$ Therefore, since we also have the inequalities
$$
P_x \left( \left\{ \limsup_{n \rightarrow +\infty} \xi_n(B) > 0 \right\} \cap \Lambda \right) \leq P_x\left( \limsup_{n \rightarrow +\infty} \xi_n(B) > 0\right) \leq P_x(\Theta),
$$ if $\eta(x)=\sigma(x)$ then we have $P_x(\Theta)=P_x(\Lambda)$ and so (ii') follows. This shows that (iii) $\Longrightarrow$ (iv). 

The implication (iv) $\Longrightarrow$ (v) is also obvious, so it remains to show (v) $\Longrightarrow$ (i). For this purpose, we note that, if $g \neq \mathbf{1}$ is a fixed point of $G$, Proposition \ref{prop:fpG} yields that $g(y) < 1$ for \textit{all} $y \in J$. Thus, by (B2) one can find $B \in \mathcal{C}_X$ with $\nu(B)>0$ and $\varepsilon > 0$ such that $\sup_{y \in B} g(y)<1-\varepsilon$. Now, since $g$ is a fixed point of $G$, we have
$$
g(x) = \lim_{n \rightarrow +\infty} G^{(n)}(g)(x) = \lim_{n \rightarrow +\infty}\E_{x} \left( \prod_{u \in \overline{\xi}_n} g(u_n) \right) \leq \liminf_{n \rightarrow +\infty} \E_{x}\left( (1-\varepsilon)^{\xi_n(B)}\right)
$$ where for the last inequality we have used the fact that $g \leq 1$. Moreover, let us observe that 
\begin{equation}
\label{eq:sscota}
\E_{x}\left( (1-\varepsilon)^{\xi_n(B)}\right) \leq P_{x}( N_n = 0) + P_{x}( \Theta^c \cap \{ N_n > 0\}) + 
\E_{x}\left( (1-\varepsilon)^{\xi_n(B)}\mathbbm{1}_\Theta\right),
\end{equation} where, since $\{N_n^{(x)} > 0\} \searrow \Theta^{(x)}$ as $n \rightarrow +\infty$, we have that 
$$
\lim_{n \rightarrow +\infty} P_{x}(N_n = 0) = \eta(x) \hspace{1cm}\text{ and }\hspace{1cm}\lim_{n \rightarrow +\infty} P_{x}(\Theta^c \cap \{N_n > 0\}) = P_{x} ( \Theta^c \cap \Theta) = 0.
$$ Hence, if we could show that 
\begin{equation}
\label{eq:sscota2}
\liminf_{n \rightarrow +\infty} \E_x\left( (1-\varepsilon)^{\xi_n(B)}\mathbbm{1}_{\Theta}\right) = 0,
\end{equation} then we would immediately obtain that $g(x) \leq \eta(x)$. Together with the obvious reverse inequality, this would yield $\eta(x)=g(x)$ and hence, by Proposition \ref{prop:fpG}, that $\eta \equiv g$. Since we have that $\eta \neq \mathbf{1}$ by the strong supercriticality of $\xi^{(x)}$, (v) $\Longrightarrow$ (i) would follow at once. Thus, let us show \eqref{eq:sscota2}.
Observe that \eqref{eq:sscota2} immediately follows if we can show that for any $K > 0$ 
\begin{equation} \label{eq:sscota3}
\liminf_{n \rightarrow +\infty} P_x ( \xi_n(B) \leq K | \Theta) = 0.
\end{equation} Since $\xi^{(x)}$ is strongly supercritical, by (ii') of Proposition \ref{prop:ss1} we have that \eqref{eq:sscota3} is then equivalent to 
\begin{equation} \label{eq:sscota4}
\liminf_{n \rightarrow +\infty} P_x \left( \xi_n(B) \leq K \Bigg| \limsup_{j \rightarrow +\infty} \xi_{j}(\tilde{J}_k) > 0\right) = 0
\end{equation} if $k \in \N$ is taken sufficiently large.
But \eqref{eq:sscota4} is now a straightforward consequence of Lemma \ref{lemma:cru}, so that \eqref{eq:sscota2} follows. 

\section{Proof of Proposition \ref{prop:lyapunov1}}\label{sec:rpos}

We conclude by proving Proposition \ref{prop:lyapunov1}, which is essentially a consequence of the following. 

\begin{proposition}\label{prop:Lyapunov} If $X$ is $\lambda$-positive and admits a Lyapunov functional $V$ as in Definition \ref{def:lyapunov}, 
	then $\mu(g) < +\infty$ for any $\B_J$-measurable $g:J \rightarrow \R$ such that $\|\frac{g}{1+V}\|_\infty < +\infty$ and, furthermore, there exist constants $C,\gamma > 0$ such that for all $x \in J$, $t \geq 0$ and any $g$ as above one has
	$$
	|\tilde{\E}_x(g(X_t)) - \mu(g)| \leq C (1+V(x)) e^{-\gamma t} \left\|\frac{g-\mu( g)}{1+V}\right\|_\infty.
	$$
\end{proposition}

\begin{proof} This result is a careful combination of \cite[Theorem 3.6]{hairer2016} and \mbox{\cite[Theorem 4.3, Theorem 6.1]{meyn1993},} using the fact that since 
	$$
	\mu(g) \leq \left\| \frac{g}{1+V}\right\|_\infty \mu(1+V) = \left\| \frac{g}{1+V}\right\|_\infty (1+ \mu(V))
	$$ it suffices to check that $\mu(V) < +\infty$ to see that $\mu(g) < +\infty$ for any function $g$ as in the statement. We omit the details.
\end{proof}

Observe that if $X$ admits a Lyapunov functional $V$ satisfying (V4) then, by Proposition \ref{prop:Lyapunov}, \mbox{if we set}
$$
\mathcal{C}_X:=\left\{ B \in \B_J : \left\|\frac{\mathbbm{1}_B}{h}\right\|_\infty <+\infty\right\}
$$ then for any $B \in \mathcal{C}_X$ and $x \in J$ we have that 
$$
|s_B(x,t)|=\left|\tilde{\E}_x\left(\frac{\mathbbm{1}_B}{h}(X_t)\right) - \mu\left(\frac{\mathbbm{1}_B}{h}\right)\right| \leq 2C(1+V(x))e^{-\gamma t}\left\|\frac{\mathbbm{1}_B}{h}\right\|_\infty
$$ so that if (V3) holds then (A2-iii) is automatically satisfied. Furthermore, we always have that
$$
s_B(x,t) \leq \tilde{\E}_x\left(\frac{\mathbbm{1}_B}{h}(X_t)\right) \leq \left\|\frac{\mathbbm{1}_B}{h}\right\|_\infty < +\infty
$$ for any such $B$, so that (A2-iv) is also satisfied. Since we have already seen that \eqref{A2} and (A2-i-ii) hold whenever $X$ is $\lambda$-positive, this shows that (A2) is satisfied. On the other hand, if (V4) holds then again by Proposition \ref{prop:Lyapunov} we have that for any $x \in J$ and $t \geq 0$ 
\begin{equation}\label{eq:cotam}
\E_x(M_t^2) = \frac{e^{\lambda t}}{h(x)} \tilde{\E}_x(h(X_t)) \leq \frac{e^{\lambda t}}{h(x)} \left( \mu(h) + 2C(1+V(x))\left( \left\|\frac{h}{1+V}\right\|_\infty + \mu(h) \right)\right) < +\infty
\end{equation} so that (A1) is also satisfied. Finally, since from \eqref{eq:cotam} we can obtain in fact the bound
$$
\E_x(M_t^2) \leq C_{h,\mu}\frac{e^{\lambda t}}{h(x)} (1+V(x))
$$ for some constant $C_{h,\mu} > 0$, a straightforward computation shows that
$$
\Phi_x \leq \frac{m_2-m_1}{r(m_1-1)-\lambda} C_{h,\mu} \cdot \frac{1+V(x)}{h(x)}.
$$ In particular, this yields that (B1) immediately holds. Furthermore, since for any given $B \in \B_J$ we have that $B \cap J_n \in \mathcal{C}_X$ for all $n$, we obtain also (B2) by taking $B^*=B \cap J_n$ for $n$ large enough so as to guarantee that $\nu(B\cap J_n) \geq \frac{\nu(B)}{2}$. This concludes the proof of Proposition \ref{prop:lyapunov1}.


\section*{Acknowledgements:}
Matthieu Jonckheere would like to warmly thank Elie A\"idekon, Julien Beresticky, Simon Harris, Pablo Groisman and Pascal Maillard for many fruitful discussions on branching dynamics. 

Santiago Saglietti would like to also thank Maria Eul\'alia Vares for useful suggestions on how to improve this manuscript and Bastien Mallein for very helpful discussions on these topics.

\bibliographystyle{plain}
\bibliography{biblio4}
 
\end{document}